\documentclass[Threeside,11pt]{article}
\usepackage [latin1]{inputenc}
\usepackage[english]{babel}
\usepackage{amsmath}
\usepackage{amsthm}
\usepackage{algorithm}
\usepackage{cite}
\usepackage{amsfonts}
\usepackage{amssymb}
\usepackage{mathrsfs}
\usepackage{graphicx}
\usepackage{colortbl,dcolumn}
\usepackage{marvosym}
\usepackage{ifsym}
\usepackage{paralist}
\usepackage{xcolor}
\usepackage[
colorlinks,
citecolor=blue,
linkcolor=blue,
backref=page
allcolors=black
]{hyperref} 
\allowdisplaybreaks
       \newtheorem{lemma}{\bf Lemma}[section]
       \newtheorem{theorem}{\bf Theorem}[section]
       \newtheorem{proposition}{\bf Proposition}[section]
       
       \newtheorem{definition}{\bf Definition}[section]
       \newtheorem{remark}{\bf Remark}[section]

       \numberwithin{equation}{section}

\begin{document}
\title{{\sl Towards sharp regularity: Full-dimensional tori in $ C^\infty $ vector fields over $ \mathbb{T}^\infty $}}
\author{Zhicheng Tong $^{\mathcal{z}}$, Yong Li $^{\mathcal{x}}$}

\renewcommand{\thefootnote}{}
\footnotetext{\hspace*{-6mm}

\begin{tabular}{l   l}	$^\mathcal{z}$~School of Mathematics, Jilin University, Changchun 130012, P. R.  China. \url{tongzc20@mails.jlu.edu.cn}\\	$^{\mathcal{x}}$~The corresponding author. School of Mathematics, Jilin University, Changchun 130012, P. R.  China; \\  Center for Mathematics and Interdisciplinary Sciences,  Northeast  Normal  University, Changchun 130024, \\P. R. China.  \url{liyong@jlu.edu.cn}
\end{tabular}}

\date{}
\maketitle

\begin{abstract}
We consider the $ C^1 $ linearization of a perturbed vector field $ \omega+P $ over the infinite-dimensional torus $ \mathbb{T}^\infty $, and determine the sharp regularity requirement of the perturbation $ P $ conjugating the unperturbed one $ \omega $ onto $ \omega-\tilde{\omega}+P $ via a small modifying term $ \tilde{\omega} $. We discuss the Diophantine type introduced by Bourgain,  investigate the universal nonresonance, and provide the weakest known  regularity of perturbations  for which KAM applies. Our results allow for lower regularity than analyticity such as  Gevrey regularity or even $ C^\infty $ regularity. We propose a new KAM scheme with a balancing sequence to overcome the non-polynomial nonresonance, differing from the usual Newtonian approach. Thereby, in addition to deriving the sharp Gevrey exponent along Diophantine nonresonance, we answer the fundamental question of what the minimum regularity required for KAM is in the infinite-dimensional setting: $ C^\infty $ regularity. Our linearization technique also applies to the quasi-periodic case over $ \mathbb{T}^n $.
\end{abstract}

\tableofcontents

\section{Introduction}
\setcounter{footnote}{0}
\renewcommand{\thefootnote}{\fnsymbol{footnote}}
Since Moser's renowned work \cite{MR0147741}, a fundamental question  in KAM theory (and indeed in dynamical systems) has been: \textit{What is the minimum KAM regularity in the infinite-dimensional setting?}
In this paper, we will address this question and also derive some other results of independent interest. Let us first review the history and advancements of KAM theory, and then delve into the  profound background of the aforementioned  problem.

 The celebrated KAM theory established by Kolmogorov and Arnold \cite{MR0068687,MR0140699,MR0163025,MR0170705}, Moser \cite{MR0147741,MR0199523,MR0206461}, mainly concerns the preservation of invariant tori of a Hamiltonian function or a vector field under small perturbations, and  has a long  history of seventy years. So far, the KAM theory has been well developed and widely applied to a variety of  dynamical systems and PDEs. On these aspects, see Kuksin \cite{MR0911772},  Eliasson \cite{MR1001032}, P\"oschel \cite{MR1037110}, Wayne \cite{MR1040892}, Bourgain   \cite{MR1316975,MR1345016}, Kuksin-P\"oschel \cite{MR1370761} and Salamon \cite{MR2111297} for some  fundamental developments. As for recent works on PDEs involving  nonlinear Schr\"odinger equation (NLS), wave equation, beam equation as well as Euler equation via KAM approaches, we particularly mention the works of Bourgain \cite{MR2180074}, Khesin \textit{et al}. \cite{MR3269186}, Eliasson \textit{et al}. \cite{MR3579706},  Biasco \textit{et al}. \cite{MR4091501,MR4666300},	Berti \textit{et al}. \cite{MR4215196}, Montalto-Procesi \cite{MR4201442}, Zhang-Si \cite{MR4273175,MR4532963}, Guardia \textit{et al}.  \cite{MR4577970},   Cong \cite{MR4800925}\footnote{The main purpose of the present paper is to weaken \textit{the KAM regularity required for perturbed vector fields} over $ \mathbb{T}^\infty $ to $ C^\infty $, following Moser's renowned motivation for the finite-dimensional case \cite{MR0147741}, and the full-dimensional (or maximal) tori obtained are only $ C^1 $; while recently, Cong's  interesting work \cite{MR4800925} extends Bourgain's work \cite{MR2180074} on NLS, obtaining $ C^\infty $ full-dimensional tori whose radius satisfies a slower decay for any $ \sigma>2 $: $ I_n \sim {\rm e}^{-\ln ^\sigma|n|} $ as  $|n| \rightarrow \infty $. Therefore, the problems dealt with and the results obtained are entirely different.} and the references therein. We also refer to the overview of KAM theory for PDEs by Kuksin \cite{MR2070057}, and  finite and infinite dimensional systems by Chierchia-Procesi  \cite{MR4570702}. 
 
 However, as commented in \cite{MR1370761},  the abstract KAM results in infinite-dimensional are very few in the light of difficulties caused by small divisors and spatial structure, which bring great difficulties to proving the persistence of  full-dimensional tori and reducing regularity for perturbations over $ \mathbb{T}^\infty $. The latter, being as one of the most fundamental questions in KAM theory, has attracted high attention in the study of \textit{how smooth the perturbation has to be to ensure the KAM persistence}. In the $ n $-dimensional case ($ 2 \leqslant n <+\infty $), it is well known that finite differentiability is sufficient, namely $ C^{2n} $ (plus certain Dini modulus of continuity) regularity for the perturbed Hamiltonian functions, and $ C^{n-1} $ regularity for the perturbed vector fields, see  Albrecht \cite{MR2350326} and P\"oschel \cite{PoarXiv} for the KAM theory, and Takens \cite{MR300311}, Herman \cite{MR0874026}, Mather \cite{MR0967638}\footnote{We would also like to mention an analogous argument by Forni \cite{MR1279471} for the  analytic case.}, Cheng-Wang \cite{MR3061774} and Wang \cite{MR4385768,WLARXIV} for counterexamples with lower regularity. It shall be emphasized that the above regularity is actually sharp, and therefore when considering infinite-dimensional situations,  finitely differentiable perturbations would destroy KAM invariant tori. As a result, one demands at least $ C^\infty $ regularity  in the almost periodic setting. Regularity usually available for KAM at this case is analyticity,  and to the best of our knowledge, there do not exist any abstract Gevrey KAM except for very few results concerning with Gevrey dependent potential in specific PDEs via the Diophantine nonresonance introduced by Bourgain in \cite{MR2180074}, namely
 \begin{equation}\label{VECDio'}
 	\left| {k  \cdot \omega } \right| > \gamma^* \prod\limits_{j \in \mathbb{Z}} {\frac{1}{{\left( {1 + {{\left| {{k _j}} \right|}^\mu }{{\left\langle j \right\rangle }^\mu }} \right)}}} ,\;\;\text{$ \forall k  \in {\mathbb{Z}^\mathbb{Z}} $ satisfying $ 0 < \sum\limits_{j \in \mathbb{Z}} {\left| {{k _j}} \right|}  <  + \infty  $}
 \end{equation}
 with $ 0<\gamma^* <1 $ and $ \mu > 1 $, and $ \left\langle j \right\rangle : = \max \left\{ {1,\left| j \right|} \right\} $ for $ j \in \mathbb{Z} $, which we refer to Biasco \textit{et al}. \cite{MR4091501} and Procesi-Stolovitch \cite{MR4456121} for relevant works. It seems that, such Diophantine frequencies are indeed related to the Gevrey regularity for perturbations over $ \mathbb{T}^\infty $ denoted by $ P $, and we mention that for given nonresonance, the minimum regularity of $ P $ should depend on it as well as certain spatial structure due to almost periodicity. Therefore, reducing regularity must base on  explicitly constructed nonresonance.

 To answer the previous addressed fundamental question on minimum KAM regularity via almost periodicity, we investigate linearization of perturbed vector fields over $ \mathbb{T}^\infty $ without Hamiltonian structure for simplicity,  and  present two main sharp results in this paper:
 \begin{itemize}
 \item With the Diophantine nonresonance in \eqref{VECDio'} it is shown that non-analytic Gevrey type regularity is sufficient to obtain KAM conjugacy, and the Gevrey exponent is indeed sharp even for the  finite-dimensional case;
 \item For almost all vector fields, certain $ C^\infty $ regularity beyond Gevery is sufficient to guarantee the persistence of full-dimensional tori, and it is also sharp in view of the fact that $ C^\infty $ regularity cannot be reduced to  finite differentiability.
 \end{itemize}

 Obviously, achieving sharp results for infinite-dimensional KAM  is a challenge. To this end, via a specific spatial structure we  extend P\"oschel's non-Newtonian KAM scheme for quasi-periodicity in \cite{PoarXiv} to the almost periodic case. In contrast to the typical super-exponential convergence, our iteration possesses an arbitrarily slow convergence rate, thereby weakening the regularity assumption to the maximum extent possible. To be more precise, we solve a nonlinear equation at each KAM step by employing the Banach contraction theorem  instead of  linearizing the conjugacy equation under consideration. An essential fact should be stressed is that P\"oschel's KAM only applies to \textit{finite-dimensional Diophantine} vector fields. Hence, we introduce a \textit{balancing sequence} to address small divisors in the infinite-dimensional context, which represents a significant distinction. As you will see, choosing an appropriate balancing sequence allows one to derive sharp regularity (not given in advance) for which KAM applies, and this is completely different from the usual cases where one has to first assume regularity (for instance, analyticity or Gevrey regularity). 
 
 \renewcommand{\thefootnote}{\fnsymbol{footnote}}
  Apart from above, instead of utilizing  the Jackson type approximation theorem in the finite-dimensional case for finitely differentiable KAM, we combine certain KAM coordinate transformations with the analytic smoothing approach to construct a special $ m $-weighted norm that reveals explicit regularity about perturbations over $ \mathbb{T}^\infty $. \textit{Thereby, we provide an effective and  universal approach to reduce regularity in the infinite-dimensional case for the first time.} Furthermore, our results are of real physical interest. For example, as Arnaiz utilized linearization theorem of the perturbed vector fields over $ \mathbb{T}^n $ to study  semiclassical KAM as well as renormalization  theorems based on counterterms (acting as the `modifying term' in our paper) in \cite{MR4105369}, certain semiclassical measures and quantum limits could be well characterized. As a consequence, our infinite-dimensional linearization theorems  would play an important role in further touching such physically  related problems in the almost periodic sense, even considering less regular symbols (see Remark 1 in \cite{MR4105369})\footnote{However, to prevent the content of this paper from becoming excessively lengthy, we prefer not to explore specific applications here, but to focus primarily on the KAM theory itself. As previously mentioned, our results can be utilized to investigate many physically  related problems.}.

The rest of the paper is organized as follows. Towards almost periodicity, we introduce nonresonance, the spatial structure and Fourier analysis as preliminaries in Section \ref{VECPreliminaries}. Then our main results, namely Quantitative Gevrey type KAM Theorem \ref{VECT1} concerning with Diophantine nonresonance, Quantitative logarithmic $ C^\infty $ type KAM Theorem \ref{VECT2} and Qualitative $ C^\infty $ type KAM Theorem \ref{VECT33} via universal nonresonance (both of them are non-Gevrey), are stated in Section \ref{VECSTEATE}, respectively. Our approach is also valid for the quasi-periodic cases and shown to be sharp, as explained by Theorem \ref{VECT3} in Section \ref{VECAPP}. Section \ref{VECPROOF} is devoted to the proofs of all main results. There, we shall  provide a guideline for readers: some basic lemmas that are crucial in the KAM theory are constructed firstly, aiming to establish an Abstract $ m $-weighted KAM Theorem \ref{VECABSTRACT} for $ C^1$ linearization; then  Theorems \ref{VECT1}, \ref{VECT2}, \ref{VECT33} and \ref{VECT3} are direct  corollaries of Theorem \ref{VECABSTRACT}, by successfully and appropriately selecting balancing sequences as already mentioned. Finally, we give further discussions in Section \ref{VECSEC6}, involving some interesting connections with Corsi-Gentile-Procesi's KAM \cite{MR4781767}, a more quantitative version of Theorem \ref{VECT33} with two  alternative approaches, as well as a  minimization idea.

\section{Preliminaries}\label{VECPreliminaries}

Let us recall the approximation function, the  infinite-dimensional Diophantine nonresonance, the weighted norm for vectors, the $ \mathbb{T}^\infty_\sigma $ torus and the  analyticity on it to be studied in this paper.   Denote by $ \left|  \cdot  \right| $ the sup-norm on the infinite-dimensional vector space $ \mathbb{R}^{\mathbb{Z}} $ (or the  finite-dimensional vector space $ \mathbb{R}^{n} $ with $ n \in \mathbb{N}^+ $). \textit{Note that the function setting in this section is set up only  to state the quantitative Gevrey type result.} In fact, we will employ a more general spatial structure to deal with $ C^\infty $ regularity.

\begin{definition}[Approximation function]\label{VECApproximationfundef}
	A function $ \Delta :\left[ {1, + \infty } \right) \to \left[ {1, + \infty } \right) $ is said to be an approximation function, if it is continuous, strictly monotonically  increasing, and satisfies  $ \Delta(+\infty) =+\infty$.
\end{definition}
Such approximation functions will be used to characterize certain universal nonresonance in the  infinite-dimensional case (see Section \ref{VECPT2}) beyond the Diophantine type below, as well as a weight that embodies the spatial structure (see Theorems \ref{VECT2} and \ref{VECT33}).

\begin{definition}[Diophantine nonresonance]\label{VECINDIO}
	Given $ \gamma^* \in (0, 1) $ and $ \mu > 1 $,   the Diophantine nonresonance of frequency $ \omega  \in {\left[ {1,2} \right]^\mathbb{Z}} $ means
	\begin{equation}\label{VECDio}
		\left| {k  \cdot \omega } \right| > \gamma^* \prod\limits_{j \in \mathbb{Z}} {\frac{1}{{\left( {1 + {{\left| {{k _j}} \right|}^\mu }{{\left\langle j \right\rangle }^\mu }} \right)}}} ,\;\;\forall k  \in {\mathbb{Z}^\mathbb{Z}},\;\;0 < \sum\limits_{j \in \mathbb{Z}} {\left| {{k _j}} \right|}  <  + \infty ,
	\end{equation}
	provided  $ \left\langle j \right\rangle : = \max \left\{ {1,\left| j \right|} \right\} $ for $ j \in \mathbb{Z} $.
\end{definition}

Let us denote by $ {\rm{D}_{\gamma^* ,\mu }} $ the set containing Diophantine  frequencies in the infinite-dimensional case. Fortunately, it was proved by Bourgain \cite{MR2180074} that  there exists a constant $ C(\mu)>0 $ such that $ \mathbb{P}\left( {{{\left[ {1,2} \right]}^\mathbb{Z}}\backslash {\rm{D}_{\gamma^* ,\mu }}} \right) \leqslant C\left( \mu  \right)\gamma^*  $, that is, almost all frequencies are of the Diophantine type in a measure-theoretical sense. 
   See more details from \cite{MR2180074,MR4456121,MR4201442,MR4091501}, for instance. 

Next, we shall introduce the infinite-dimensional torus and the Fourier expansion on it. For $ \eta,\sigma>0 $,   the thickened infinite-dimensional torus is define by
\[\mathbb{T}_\sigma ^\infty : = \left\{ {x = {{({x_j})}_{j \in \mathbb{Z}}},{x_j} \in \mathbb{C}:\operatorname{Re} {x_j} \in \mathbb{T},\left| {\operatorname{Im} {x_j}} \right| \leqslant \sigma {{\left\langle j \right\rangle }^\eta }}, j \in \mathbb{Z} \right\}.\]
In particular, it represents the usual torus $ \mathbb{T}^\infty $ when $ \sigma  $ degenerates to $ 0 $. 
For given $ \eta >0 $, we define the set of infinite integer vectors with finite support 
\[\mathbb{Z}_ * ^\infty : = \left\{ {k  \in {\mathbb{Z}^\mathbb{Z}}:{{\left| k  \right|}_\eta }: = \sum\limits_{j \in \mathbb{Z}} {{{\left\langle j \right\rangle }^\eta }\left| {{k _j}} \right|}  <  + \infty } \right\}.\]
Such a spatial structure like this is not necessary in the finite-dimensional case, but cannot be removed  in the  infinite-dimensional setting; otherwise, the Fourier series would diverge. It is worth mentioning that a similar spatial structure  appears in Theorems \ref{VECT2} and \ref{VECT33}. Via the above notations, the analyticity on $ \mathbb{T}^\infty_\sigma $ with the exponential weighted norm could be given.
\begin{definition}\label{VECcijiexi}
	For  $ \sigma > 0 $,  the Banach space $ \mathcal{G}\left( {\mathbb{T}_\sigma ^\infty } \right) $ containing analytic functions or maps on $ \mathbb{T}_\sigma ^\infty $ is defined by
	\[\mathcal{G}\left( {\mathbb{T}_\sigma ^\infty } \right): = \left\{ {u\left(x  \right) = \sum\limits_{k  \in \mathbb{Z}_ * ^\infty } {\hat u\left( k  \right){{\rm e}^{{\rm i}k  \cdot x }}} :\|u\|_{\sigma }: = \sum\limits_{k  \in \mathbb{Z}_ * ^\infty } {\left| {\hat u\left( k  \right)} \right|{{\rm e}^{\sigma {{\left| k  \right|}_\eta }}}}  <  + \infty } \right\}.\]
	
\end{definition}
\begin{remark}\label{sanjiao}
It follows from the definition of the norm $ {\left\|  \cdot  \right\|_\sigma } $ that $ {\left\| {{{\rm e}^{{\rm i}k \cdot x}}} \right\|_\sigma } = {{\rm e}^{\sigma{{\left| k \right|}_\eta }}} $ for $ k \in \mathbb{Z}_ * ^\infty  $, and $   {\left\| {X } \right\|_\sigma }={\left\| {X } \right\|_0}=\left|{X }\right| $ holds for  $X $ being constant.
\end{remark}

Given the analyticity on $ \mathbb{T}^\infty_\sigma $ as defined in Definition \ref{VECcijiexi}, we have the following proposition, which has been discussed by Montalto-Procesi \cite{MR4201442}, and also in a more general fashion by Corsi \textit{et al}. \cite{MR4781767}.
\begin{proposition}
For $ u \in \mathcal{G}\left( {\mathbb{T}_\sigma ^\infty } \right)  $ with $ \sigma>0 $, we have
\[\int_{{\mathbb{T}^\infty }} u (x){\rm d}x: = \mathop {\lim }\limits_{N \to  + \infty } \frac{1}{{{{(2\pi )}^{2N + 1}}}}\int_{{\mathbb{T}^{2N + 1}}} u (x){\rm d}{x_{ - N}} \ldots {\rm d}{x_0} \ldots {\rm d}{x_N} = \hat u(0).\]
Moreover, for any $ 0 \ne k \in \mathbb{Z}_ * ^\infty  $, we have
\[\hat u(k) = \int_{{\mathbb{T}^\infty }} u (x){e^{ - {\rm{i}}k \cdot x}}{\rm d}x = \mathop {\lim }\limits_{N \to +\infty } \frac{1}{{{{(2\pi )}^{2N + 1}}}}\int_{{\mathbb{T}^{2N + 1}}} u (x){e^{ - {\rm{i}}k \cdot x}}\ldots {\rm d}{x_0} \ldots {\rm d}{x_{ - N}} \ldots {\rm d}{x_N}.\]
\end{proposition}

Finally, let us introduce a special $ m $-weighted norm that will be used to construct the regularity of perturbations as weak as possible in KAM. We say that $m = {\left\{ {{m_k}} \right\}_{k \in {\mathbb{Z}_*^\infty}}} $ is  a weight, if  $ {m_k} = m( {\left| k \right|_\eta} ) \geqslant 0 $ is non-decreasing for every $ k \in \mathbb{Z}_*^\infty $. We always denote by $ m $ a weight rather than a non-negative number (such as $ \sigma $ in Definition \ref{VECcijiexi}) throughout this paper. Based on a weight $ m $, one could consider a certain $ m $-weighted norm from our non-Newtonian KAM that has weaker regularity than analyticity (such as Gevrey regularity or even $ C^\infty $ regularity).
\renewcommand{\thefootnote}{\fnsymbol{footnote}} 
\begin{definition}
Consider a map with Fourier expansion series $ f = \sum\nolimits_{k \in \mathbb{Z}_ * ^\infty } {\hat f(k){{\rm e}^{{\mathrm{i}}k \cdot x}}}  $.  Then for a weight $ m = {\left\{ {{m_k}} \right\}_{k \in \mathbb{Z}_ * ^\infty }} $, define the $ m $-weighted norm of $ f $ as\footnote{We reaffirm that the norm of the Fourier coefficients here (i.e., $ | {\hat f(k)} | $) is defined as the sup-norm of an infinite-dimensional vector, as specified at the beginning of Section \ref{VECPreliminaries}. However, it is noteworthy that the sup-norm employed here can be substituted with any norm applicable to finite-dimensional vectors, because for any given $ k \in \mathbb{Z}_*^\infty $, it possesses compact support.}
\[{\left\| f \right\|_m}: = \sum\limits_{k \in \mathbb{Z}_ * ^\infty } {| {\hat f(k)} |{m_k}} .\]
\end{definition}
\begin{remark}
For instance, $ {m_k} = \left| k \right|_\eta ^a  $ corresponds to the Gevery regularity for $ 0<a<1 $ to analyticity for $ a=1 $, respectively. We also refer to the finite-dimensional Gevrey regularity proposed by Popov \cite{MR2104602}, which is equivalent in many situations.
\end{remark}

\section{Main results: The almost periodic case}\label{VECSTEATE}
As usual, let us first discuss the Diophantine vector field in the almost periodic case. It is shown that, assuming the Gevrey type regularity depending on the Diophantine nonresonance for perturbations is sufficient to ensure the KAM linearization via modifying terms. Compared with the well known  fact in the finite-dimensional case, that is, the critical regularity in KAM theory strongly depends on the Diophantine exponent admitted by the frequency (if it indeed has, for instance, see Salamon \cite{MR2111297} and the authors \cite{TLCCM}), our infinite-dimensional result also reveals the similar  point.  Moreover, the  Gevrey exponent obtained below is indeed \textit{sharp}, even in the quasi-periodic case, see Section \ref{VECAPP} for explanations. 

\renewcommand{\thefootnote}{\fnsymbol{footnote}} 
\begin{theorem}[\textbf{Quantitative Gevrey type  KAM}]\label{VECT1}	
Let $ \eta>0 $ be given. Assume that $ \omega \in \mathbb{T}^\infty $ satisfies the infinite-dimensional Diophantine condition in Definition \ref{VECINDIO}. Then there exists some $ C_{{\rm{G1}}} >0$, as long as the perturbation $ P $ is sufficiently small in the sense that\footnote{The Gevrey exponent in the first version of this work \cite{arXiv:2306.08211} was weaker than it is now, namely   $ \frac{1}{1+\eta'} $ with arbitrary $0< \eta'<\eta $, but without the present logarithmic term. We have further optimized it by quantitatively constructing a balancing sequence.}
\[{\left\| P \right\|_m} = \sum\limits_{0 \ne k \in \mathbb{Z}_*^\infty } {|\hat P(k)|\exp \left( {{C_{{\rm{G1}}}}\ln \left( {1 + {{\left| k \right|}_\eta }} \right)\left| k \right|_\eta ^{\frac{1}{{1 + \eta }}}} \right)}  \ll 1,\]
there exist a modifying term $ \tilde \omega $ and a nearly identical diffeomorphism $ \Psi $ such that $ \omega -\tilde \omega + {P } $ is conjugated to $ \omega $, i.e.,
\[	\Psi^ * \left( {\omega -\tilde \omega + {P } } \right) = \omega .\]
\end{theorem}

Unlike the discrete case (in map form), for the continuous vector field above, $ \Psi^ * \left( {\omega -\tilde \omega + {P } } \right) = \omega $ implies 
$ {\left( {D\Psi } \right)^{ - 1}}\left( {\omega  - \tilde \omega  + P} \right) \circ \Psi  = \omega  $, which is commonly referred to as the linearization of the vector field in dynamical systems, see Bounemoura \cite{MR4412075} and  Koch \cite{MR4735770},	for instance. The concept of the modifying term $ \tilde{\omega} $ was initially introduced by Arnold \cite{MR0140699} within  finite-dimensional systems  and later generalized by Moser \cite{MR0208078}. Notably, it is indeed close to $ 0 $, allowing the prescribed frequency  to drift. Specifically, if the rotation set of the perturbed vector field includes  $ \omega $, then the modifying term $ \tilde{\omega} $ vanishes, which means that we achieve frequency-preserving; however, this is not generally the case. The former situation  corresponds to the discrete linearization case.  As for the nearly identical diffeomorphism $ \Psi $, it is indeed unique upon normalization in the sense that the average of $ \Psi  -{\rm id} $ is zero. Both the modifying term $ \tilde{\omega} $ and the nearly identical diffeomorphism $ \Psi $ are nonlinear with respect to the perturbation $ P $. For further details, see Bounemoura \cite{MR4412075}, for instance.

\renewcommand{\thefootnote}{\fnsymbol{footnote}} 
Now we consider the nonresonance beyond Diophantine, for which almost all vector fields hold, and we say that it is \textit{universal}. One will see later that the \textit{non-Gevrey} regularity for perturbations is explicitly shown, but we shall emphasize, it could be weaker, whenever the weight added in \eqref{VECT2RE} is larger than any polynomial's form, see details from Theorem \ref{VECT33}. This fact reveals the \textit{sharpness}, since in the almost periodic case, the polynomial's nonresonance or regularity \textit{does not} ensure the KAM persistence, which we refer to the  counterexamples constructed by Takens \cite{MR300311}, Herman \cite{MR0874026}, Mather \cite{MR0967638},  Cheng-Wang \cite{MR3061774}, Wang \cite{MR4385768,WLARXIV} and the references therein (as the dimension tends to infinity)\footnote{Please note that this paper does not claim to conduct work in the field of converse KAM theory; it simply references historical results.}. \textit{Thereby, Theorem \ref{VECT33} provides the weakest regularity---namely $ C^\infty $ for which  KAM  applies in the universal sense over $ \mathbb{T}^\infty $. Consequently, it fills the gap in the results concerning this aspect.}

\begin{theorem}[\textbf{Quantitative logarithmic $ C^\infty $ type KAM}]\label{VECT2}
For almost all $ \omega \in \mathbb{R}^{\mathbb{Z}} $, there exists a uniform approximation function $ w $,   as long as  the perturbation $ P $ is sufficiently small in the sense that
\begin{equation}\label{VECT2RE}
	\sum\limits_{0 \ne k \in \mathbb{Z}_ * ^\infty } {|\hat P(k)|\exp \left(C_{\rm L1} {{{\left( {\ln {{\left\| k \right\|}_w}} \right)}^a}} \right)}  \ll 1,
\end{equation}
where $ {\left\| k \right\|_w}: = \sum\nolimits_{j \in \mathbb{Z}} {w\left( {\left\langle j \right\rangle } \right)\left| {{k_j}} \right|}  $ and $ C_{\rm L1}>0,a>1 $,  the KAM conjugacy in Theorem \ref{VECT1} holds.
\end{theorem}
\begin{remark}\label{VECDUISHUGZ}
Theorem \ref{VECT2RE} is based on the following  universal nonresonence with any $ \tilde{C}_{\rm L1}>0 $ sufficiently small:
\[\left| {k \cdot \omega } \right| > {\gamma ^*}\exp \left( { - \tilde{C}_{\rm L1} {{\left( {\ln {{\left\| k \right\|}_w}} \right)}^a}} \right),\;\;{\gamma ^*} > 0,\;\;0 \ne k \in \mathbb{Z}_*^\infty ,\]
and, consequently, it is indeed a quantitative result.
\end{remark}

Building on certain weaker nonresonance than that outlined  in Remark \ref{VECDUISHUGZ}, we establish the following main result:

\renewcommand{\thefootnote}{\fnsymbol{footnote}}
\begin{theorem}[\textbf{Qualitative $ C^\infty $ type  KAM}]\label{VECT33}
	For almost all $ \omega \in \mathbb{R}^{\mathbb{Z}} $ and any given approximation function $ \Delta $ larger than any polynomial's type\footnote{That is, $ \Delta(x)\gg x^L $ for arbitrary $ L>0 $, and we say that $ \Delta $ is `super-polynomial'.}, there exists a uniform approximation function $ u $ depending on $ \Delta $,   as long as  the perturbation $ P $ is sufficiently small in the sense that
\[\sum\limits_{0 \ne k \in \mathbb{Z}_ * ^\infty } {|\hat P(k)|\Delta({\left\| k \right\|}_{u})}  \ll 1,\]
	where $ {\left\| k \right\|_u}: = \sum\nolimits_{j \in \mathbb{Z}} {u\left( {\left\langle j \right\rangle } \right)\left| {{k_j}} \right|}  $,  the KAM conjugacy in Theorem \ref{VECT1} holds.
\end{theorem}
\begin{remark}
Theorem \ref{VECT33} is inspired by Moser's renowned work \cite{MR0147741} on $ C^{333} $ regularity in KAM theory, and thus we are inclined to refer to it as a Moser-type result. It delivers one of the most fundamental conclusions in KAM theory: $ C^\infty $ regularity can ensure the validity of the infinite-dimensional KAM theory. 
\end{remark}
\begin{remark}
In addition to the proof given in Section \ref{VECPT2},    Section \ref{VECSEC62} will provide two alternative approaches to Theorem \ref{VECT33}, particularly in a quantitative way.
\end{remark}

It should be noted that Theorems \ref{VECT1} to \ref{VECT33} are all corollaries of a more abstract but more refined Theorem \ref{VECABSTRACT} presented in Section \ref{VECABST}, which provides $C^1$ linearization. In the context of specific regularity applications, the appropriate selection of the balancing sequence may lead to a conjugacy regularity higher than $C^1$. This implies that the tori obtained in Theorems \ref{VECT1} to \ref{VECT33} may possess higher regularity (e.g., $C^\infty$ or even higher).  However, for the sake of brevity, we prefer not to delve into this point in this paper.

\section{Applications to the quasi-periodic case}\label{VECAPP}
It should be noted that our result, specifically the Abstract $m$-weighted KAM Theorem \ref{VECABSTRACT} detailed in Section \ref{VECABST}, is also applicable to the quasi-periodic case with sharp regularity outlined below. The proof of this application will be discussed in Section \ref{VECYOUXIAN}.

\begin{theorem}\label{VECT3}
Let $ n \in \mathbb{N}^+ $ and an approximation function $ \Delta  $ be given.	Assume that $ \omega \in \mathbb{R}^n $ satisfies the finite-dimensional nonresonant condition
\[\left| {k \cdot \omega } \right| > \frac{{{\gamma ^ * }}}{{\Delta \left( {\left| k \right|} \right)}},\;\;0 \ne k \in {\mathbb{Z}^n},\;\;{\gamma ^ * } > 0.\]
Suppose that any of the followings is satisfied: 
\begin{itemize}
\item[(i)] $ \Delta $ is Diophantine, i.e., 
	\begin{equation}\label{VECFINITEDIO}
	\Delta(x)\sim x^\beta,\;\;\beta  \geqslant \left\{ \begin{aligned}
		&1,&n = 1,  \hfill \\
		&n-1,&n \geqslant 2, \hfill \\
	\end{aligned}  \right.
\end{equation}
and the perturbation $ P $ admits the finite differentiability regularity  as
\[\sum\limits_{0 \ne k \in {\mathbb{Z}^n}} {|\hat P(k)|{{\left| k \right|}^{\beta  + 1}}}  \ll 1 ;\]
\item[(ii)] $ \Delta $ is  sub-exponential, i.e., $ \Delta \left( x \right) \sim \exp ( {{x^\zeta}} ) $ for some $ 0<\zeta<1 $,  and the perturbation $ P $ admits the Gevrey regularity  as
\[ \sum\limits_{0 \ne k \in {\mathbb{Z}^n}} {|\hat P(k)|\exp ( C_{\rm G2}{{|k|^{\zeta }}} )}  \ll 1 \]
 with some $ C_{\rm G2} >0 $ sufficiently large;
\item[(iii)] $ \Delta \left( x \right) \sim \exp \left( {{{\left( {\ln x} \right)}^a}} \right) $ with some $ a>1 $,  and the perturbation $ P $ admits the same type regularity as
\[ \sum\limits_{0 \ne k \in {\mathbb{Z}^n}} {|\hat P(k)|\exp \left( {C_{\rm L2}{{\left( {\ln |k|} \right)}^a}} \right)}  \ll 1  ;\]
 with any $ C_{\rm L2}>1 $.
\end{itemize}
	Then there exist a modifying term $ \tilde \omega $ and a nearly identical diffeomorphism $ \Psi $ such that $ \omega -\tilde \omega + {P } $ is conjugated to $ \omega $, i.e.,
	\[	\Psi^ * \left( {\omega -\tilde \omega + {P } } \right) = \omega .\]
\end{theorem}

Note that (ii) and (iii) are the corresponding versions of Theorems \ref{VECT1} and \ref{VECT2} in the quasi-periodic case. When $ n \geqslant 2 $ in (i), although our regularity is higher than P\"oschel's in his preprint \cite{PoarXiv} (our method is somewhat different from his, under which $ n \geqslant 2 $ is necessary---due to R\"{u}ssmann's estimates in \cite{MR426054}), it is still nearly  sharp, thanks to the counterexamples of Herman \cite{MR0874026} and etc. To be more precise, $ \beta =n-1 $ with $ n \geqslant2 $ in the requirement $ \sum\nolimits_{0 \ne k \in {\mathbb{Z}^n}} {|\hat P(k)|{{\left| k \right|}^{\beta  + 1}}}  \ll 1 $ only yields the $ C^n $ regularity for the perturbation $ P $ (it cannot ensure the existence of the $ (n+1) $-th derivatives), while $ C^{n-1} $ regularity is the `critical' case (the $ (n-1) $-th derivatives may need some Dini type modulus of continuity as observed in \cite{MR2350326,TLCCM}, and $ C^{n-1} $ cannot be weakened to $ C^{n-1-\epsilon} $ for any $ \epsilon>0 $). Our results also cover the case $ n=1 $. So far we have discussed the Diophantine nonresonance in all dimensions, that is, from $ 1 $ to $ +\infty $. Besides considering the Diophantine nonresonance,  both of (ii) and (iii) are \textit{sharp} and concise, see Bounemoura \cite{MR4412075} for the optimal Gevrey instance about vector fields (our sharpness is reflected in the Gevrey exponent $ \zeta $ of the perturbation, namely based on the given nonresonance,  $ \zeta $ in the regularity requirement $ \sum\nolimits_{0 \ne k \in {\mathbb{Z}^n}} {|\hat P(k)|\exp ( C_{\rm G2}{{|k|^{\zeta }}} )}  \ll 1  $ cannot be replaced with any $ \zeta'<\zeta $), and see  Salamon \cite{MR2111297}, Koudjinan \cite{MR4104457} and the authors \cite{TLCCM} for sharp cases considering  Hamiltonian systems via universal nonresonance. Here, we would also like  to mention the works of Khesin \textit{et al.} \cite{MR3269186}, Li-Shang \cite{MR3960504} and Hu \cite{MR4623276} via finite differentiability.  However, we do not yet know what KAM analogues are based on extremely weak nonresonant conditions (for instance, \cite{MR4836959}), and these will be the subject of our future research.

\section{Proofs of the results}\label{VECPROOF}
All proofs of the results are presented in this section. However, before delving into the proofs, it is essential to establish in detail some fundamental lemmas based on the infinite-dimensional structure introduced earlier. Throughout this paper, we consistently use the notation $ {\rm id} $ to represent the identity operator, and we denote by $ \mathbb{I} $ the infinite-dimensional identity matrix, such that $ D\left( {\rm id} \right) = \mathbb{I} $. Additionally, the symbols $ \vee $ and $  \wedge  $ represent the maximum and minimum operators, respectively.

\subsection{Some basic lemmas}\label{VECLEMMAS}
\begin{lemma}[Neumann Lemma]\label{VECL1}
Let $ \varphi    $ be a self map the torus $ \mathbb{T}^\infty $. If $ \mu  = {\left\| {D\varphi  - \mathbb{I}} \right\|_\sigma } < 1 $ with $ \sigma>0 $, then $ \varphi $ is a diffeomorphism, and its inverse $ \psi  $ satisfies
\[{\left\| {D\psi } \right\|_\sigma } \leqslant \frac{1}{{1 - \mu }},\;\;{\left\| {D\psi  - \mathbb{I}} \right\|_\sigma } \leqslant \frac{\mu }{{1 - \mu }}.\]
\end{lemma}
\begin{proof}
It is evident from the Neumann type argument.
\end{proof}

\begin{lemma}[Transformation Lemma]\label{VECL2}
Consider a diffeomorphism $ \varphi $ of the torus $ \mathbb{T}^\infty $ extending to $ \mathbb{T}^\infty_\sigma $ with $ \sigma >0 $. If $ {\left\| {\varphi  - {\rm id}} \right\|_\sigma } \leqslant a $ for some $ a\geqslant0 $, then for $ f \in \mathcal{G}\left( {\mathbb{T}_{\sigma+a} ^\infty } \right) $, it holds that
\[{\left\| {f \circ \varphi } \right\|_\sigma } \leqslant {\left\| f \right\|_{\sigma  + a}}.\]
\end{lemma}
\begin{proof}
With
\[f \circ \varphi \left( x \right) = \sum\limits_{k \in \mathbb{Z}_ * ^\infty } {\hat f\left( k  \right)\exp \left( {{\rm i}k  \cdot x} \right) \cdot \exp \left( {k  \cdot \left( {\varphi \left( x \right) - {\rm id}} \right)} \right)}, \]
we have
\[{\left\| {f \circ \varphi } \right\|_\sigma } \leqslant \sum\limits_{k \in \mathbb{Z}_ * ^\infty } {| {\hat f\left( k  \right)} |{{\left\| {\exp \left( {{\rm i}k  \cdot x} \right)} \right\|}_\sigma } \cdot {{\left\| {\exp \left( {k  \cdot \left( {\varphi \left( x \right) - {\rm id}} \right)} \right)} \right\|}_\sigma }} .\]
By Remark \ref{sanjiao}, one obtains that $ {\left\| {\exp \left( {{\rm i}k  \cdot x} \right)} \right\|_\sigma } = {{\rm e}^{\sigma {{\left| k  \right|}_\eta }}} $. On the other hand, with $ \left| k  \right| \leqslant {\left| k  \right|_\eta } $ and the triangle inequality,  we can prove that
\begin{align*} 
{\left\| {\exp \left( {k  \cdot \left( {\varphi \left( x \right) - {\rm id}} \right)} \right)} \right\|_\sigma } &\leqslant \sum\limits_{j = 0}^\infty  {\frac{1}{{j!}}\left\| {k  \cdot \left( {\varphi \left( x \right) - {\rm id}} \right)} \right\|_\sigma ^j}  \leqslant \sum\limits_{j = 0}^\infty  {\frac{{\left| k  \right|_\eta ^j}}{{j!}}\left\| {\varphi \left( x \right) - {\rm id}} \right\|_\sigma ^j} \\
& \leqslant \sum\limits_{j = 0}^\infty  {\frac{{{a^j}\left| k  \right|_\eta ^j}}{{j!}}}  = {{\rm e}^{a{{\left| k  \right|}_\eta }}}.
\end{align*}
Therefore, it follows that
\[{\left\| {f \circ \varphi } \right\|_\sigma } \leqslant \sum\limits_{k \in \mathbb{Z}_ * ^\infty } {| {\hat f\left( k  \right)} |{{\rm e}^{\left( {\sigma  + a} \right){{\left| k  \right|}_\eta }}}}  = {\left\| f \right\|_{\sigma  + a}},\]
as desired.
\end{proof}

\begin{lemma}[Cauchy's estimate Lemma]\label{VECL3}
Let $ f,\varphi\in  \mathcal{G}\left( {\mathbb{T}_\sigma ^\infty } \right) $ be given, where $ \sigma>0 $. Then for $ 0<\alpha<1 $, it holds that
\[{\left\| {Df \cdot \varphi } \right\|_{\alpha \sigma}} \leqslant \frac{1}{{{\rm e}\left( {1 - \alpha } \right)\sigma}}{\left\| f \right\|_\sigma}{\left\| \varphi  \right\|_{\alpha\sigma}}.\]
\end{lemma}
\begin{proof}
It is evident that
\[Df \cdot \varphi  = {\rm i}\sum\limits_{k \in \mathbb{Z}_ * ^\infty } {\left( {k \cdot \varphi } \right)\hat f\left( k \right){{\rm e}^{{\rm i}k \cdot x}}}  = {\rm i}\sum\limits_{k \in \mathbb{Z}_ * ^\infty } {\sum\limits_{\ell  \in \mathbb{Z}_ * ^\infty } {\left( {k \cdot \hat \varphi \left( \ell  \right)} \right)\hat f\left( k \right){{\rm e}^{{\rm i}\left( {k + \ell } \right) \cdot x}}} } .\]
Recalling Remark \ref{sanjiao}, we have
\[{\left\| {\exp \left( {{\rm i}\left( {k + \ell } \right) \cdot x} \right)} \right\|_{\alpha \sigma}} = \exp \left( {\alpha \sigma{{\left| {k + \ell } \right|}_\eta }} \right) \leqslant \exp \left( {\alpha \sigma{{\left| k \right|}_\eta } + \alpha \sigma{{\left| \ell  \right|}_\eta }} \right) = {{\rm e}^{\alpha \sigma{{\left| k \right|}_\eta }}} \cdot {{\rm e}^{\alpha \sigma{{\left| \ell  \right|}_\eta }}}.\]
Therefore, by $ \left| k \right| \leqslant {\left| k \right|_\eta } $, one derives the conclusion as
\begin{align}
{\left\| {Df \cdot \varphi } \right\|_{\alpha \sigma}} &\leqslant \sum\limits_{k \in \mathbb{Z}_ * ^\infty } {\sum\limits_{\ell  \in \mathbb{Z}_ * ^\infty } {\left| k \right|| {\hat \varphi \left( \ell  \right)} || {\hat f\left( k \right)} |{{\left\| {\exp \left( {{\rm i}\left( {k + \ell } \right) \cdot x} \right)} \right\|}_{\alpha \sigma}}} } \notag \\
& \leqslant \sum\limits_{k \in \mathbb{Z}_ * ^\infty } {\sum\limits_{\ell  \in \mathbb{Z}_ * ^\infty } {{{\left| k \right|}_\eta }\left| {\hat \varphi \left( \ell  \right)} \right|| {\hat f\left( k \right)} |{{\rm e}^{\alpha \sigma{{\left| k \right|}_\eta }}} \cdot {{\rm e}^{\alpha \sigma{{\left| \ell  \right|}_\eta }}}} }\notag \\
& \leqslant \left( {\sum\limits_{k \in \mathbb{Z}_ * ^\infty } {{{\left| k \right|}_\eta }{{\rm e}^{ - \left( {1 - \alpha } \right)\sigma{{\left| k \right|}_\eta }}} \cdot | {\hat f\left( k \right)} |{{\rm e}^{\sigma{{\left| k \right|}_\eta }}}} } \right) \cdot \left( {\sum\limits_{\ell  \in \mathbb{Z}_ * ^\infty } {\left| {\hat \varphi \left( \ell  \right)} \right|{{\rm e}^{\alpha \sigma{{\left| \ell  \right|}_\eta }}}} } \right)\notag\\
\label{VECSUP1}& \leqslant \frac{1}{{{\rm e}\left( {1 - \alpha } \right)\sigma}}{\left\| f \right\|_\sigma}{\left\| \varphi  \right\|_{\alpha\sigma}},
\end{align}
where the following trivial fact is employed in \eqref{VECSUP1}:
\[\mathop {\sup }\limits_{k \in \mathbb{Z}_ * ^\infty } \left( {{{\left| k \right|}_\eta }{{\rm e}^{ - \left( {1 - \alpha } \right)\sigma{{\left| k \right|}_\eta }}}} \right) \leqslant \mathop {\sup }\limits_{t \geqslant 0} \left( {t{{\rm e}^{ - \left( {1 - \alpha } \right)st}}} \right) = \frac{1}{{{\rm e}\left( {1 - \alpha } \right)\sigma}}.\]
\end{proof}

Let a map $f \in \mathcal{G}\left( {\mathbb{T}_\sigma ^\infty } \right)$ be given. Then for $K \in \mathbb{N}$, let us define the truncation map $\mathcal{T}_K$ as well as the residual map $\mathcal{R}_{K}$ as 
\[\mathcal{T}_Kf = \sum\limits_{ \left| k \right|_\eta \leqslant K} {{\hat f\left( k \right)}{{\rm e}^{{\rm i} k \cdot x }}} ,\;\;\mathcal{R}_{K}f = \sum\limits_{ \left| k \right|_\eta > K} {{\hat f\left( k \right)}{{\rm e}^{{\rm i}k \cdot x }}} ,\]
respectively. Therefore,  $\mathcal{T}_K+\mathcal{R}_{K}={\rm id}$ holds for all $ K \in \mathbb{N} $. In particular,  $\mathcal{T}_K+\mathcal{R}_{K}={\rm id}$ for all maps $f$ without Fourier constants and $ K\in \mathbb{N}^+ $. 

\begin{lemma}[Cauchy type Lemma]\label{VECCAULE}
Let $f \in \mathcal{G}\left( {\mathbb{T}_{\sigma} ^\infty } \right)$ without the  Fourier constant term be given, where $ \sigma>0 $. Then  it holds that
\[{\left\| {D{\mathcal{T}_k}f} \right\|_{\alpha \sigma}} \leqslant K{\left\| {{\mathcal{T}_k}f} \right\|_{\alpha \sigma}}.\]
Moreover, we have
\[{\left\| {Df} \right\|_0} \leqslant \frac{1}{{\rm e}\sigma }{\left\| f \right\|_\sigma }.\]
\end{lemma}
\begin{proof}
Using $ \left| k \right| \leqslant {\left| k \right|_\eta } $ and recalling Remark \ref{sanjiao}, one directly proves that
\begin{align*}
{\left\| {D{\mathcal{T}_k}f} \right\|_{ \sigma}} &\leqslant \sum\limits_{0 < {{\left| k \right|}_\eta }  \leqslant K} {\left| k \right||\hat F\left( k \right)|{{\| {{{\rm e}^{{\rm i}k \cdot x}}} \|}_{ \sigma}}}  \leqslant \sum\limits_{0 < {{\left| k \right|}_\eta } \leqslant K} {{{\left| k \right|}_\eta }|\hat F\left( k \right)|{{\rm e}^{ \sigma{{\left| k \right|}_\eta }}}} \\
& \leqslant K\sum\limits_{0 < {{\left| k \right|}_\eta } \leqslant K} {|\hat F\left( k \right)|{{\rm e}^{ \sigma{{\left| k \right|}_\eta }}}}  = K{\left\| {{\mathcal{T}_k}f} \right\|_{ \sigma}}.
\end{align*}
On the other hand, by $ \left| k \right| \leqslant {\left| k \right|_\eta } $ and ${\rm e} t \leqslant {{\rm e}^t} $, it is shown that
\[{\left\| {Df} \right\|_0} \leqslant \sum\limits_{k \in \mathbb{Z}_ * ^\infty } {{{\left| k \right|}_\eta }|\hat f(k)|}  \leqslant \frac{1}{{\rm e}\sigma }\sum\limits_{k \in \mathbb{Z}_ * ^\infty } {|\hat f(k)|{{\rm e}^{\sigma {{\left| k \right|}_\eta }}}}  = \frac{1}{{\rm e}\sigma }{\left\| f \right\|_\sigma }.\]
\end{proof}

\begin{lemma}[Small divisor Lemma]\label{VECL4}
Let a nonresonant frequency $\omega \in \mathbb{T}^\infty$ be given. Then for $f \in \mathcal{G}\left( {\mathbb{T}_\sigma ^\infty } \right)$ without the Fourier constant term, where $ \sigma>0 $, the unique solution to the homological truncation equation 
\begin{equation}\label{VECHOMO}
	\omega  \cdot \partial g = \mathcal{T}_Kf,\;\;g: = \mathscr{L}\mathcal{T}_Kf
\end{equation}
satisfies the estimate
\[{\left\| {\mathscr{L}{\mathcal{T}_K}f} \right\|_\sigma} \leqslant \left( {\mathop {\max }\limits_{0 < {{\left| k \right|}_\eta } \leqslant K} {{\left| {k \cdot \omega } \right|}^{ - 1}}} \right){\left\| f \right\|_\sigma}.\]
\end{lemma}
\begin{proof}
The unique solution to the homological truncation equation \eqref{VECHOMO} can be written as
\[\mathscr{L}{\mathcal{T}_K}f = \sum\limits_{0 \ne {{\left| k \right|}_\eta } \leqslant K} {\frac{{\hat f\left( k \right)}}{{{\rm i}k \cdot \omega}}{{\rm e}^{{\rm i}k \cdot x}}} .\]
As a consequence, one concludes 
\begin{align*}
{\left\| {\mathscr{L}{\mathcal{T}_K}f} \right\|_\sigma} &\leqslant \sum\limits_{0 \ne {{\left| k \right|}_\eta } \leqslant K} {\frac{{| {\hat f\left( k \right)} |}}{{\left| {k \cdot \omega} \right|}}{{\| {{{\rm e}^{{\rm i}k \cdot x}}} \|}_\sigma}}  \leqslant \left( {\mathop {\max }\limits_{0 < {{\left| k \right|}_\eta } \leqslant K} {{\left| {k \cdot \omega } \right|}^{ - 1}}} \right)\left( {\sum\limits_{0 \ne {{\left| k \right|}_\eta } \leqslant K} {| {\hat f\left( k \right)} |{{\rm e}^{\sigma{{\left| k \right|}_\eta }}}} } \right)\\
& = \left( {\mathop {\max }\limits_{0 < {{\left| k \right|}_\eta } \leqslant K} {{\left| {k \cdot \omega } \right|}^{ - 1}}} \right){\left\| {{\mathcal{T}_K}f} \right\|_\sigma}
\end{align*}
due to Remark \ref{sanjiao}, which proves the lemma.
\end{proof}

\begin{lemma}[Residual term Lemma]\label{VECL5}
Let $f \in \mathcal{G}\left( {\mathbb{T}_\sigma ^\infty } \right)$ be given, where $ \sigma>0 $. Then for $ 0<\alpha <1 $, it holds that
\[{\left\| \mathcal{R}_{K}f \right\|_{\alpha \sigma }} \leqslant {{\rm e}^{ - \left( {1 - \alpha } \right)\sigma K}}{\left\| \mathcal{R}_{K}f \right\|_\sigma }.\]
\end{lemma}
\begin{proof}
With Remark \ref{sanjiao}, we prove that
\begin{align*}
{\left\| {{\mathcal{R}_K}f} \right\|_{\alpha \sigma }} &\leqslant \sum\limits_{{{\left| k \right|}_\eta } > K} {|\hat f\left( k \right)|{{\| {{{\rm e}^{{\rm i}k \cdot x}}} \|}_{\alpha \sigma }}}  \leqslant \sum\limits_{{{\left| k \right|}_\eta } > K} {|\hat f\left( k \right)|{{\rm e}^{\alpha \sigma {{\left| k \right|}_\eta }}}} \\
& \leqslant \left( {\mathop {\sup }\limits_{{{\left| k \right|}_\eta } > K} {{\rm e}^{\left( {\alpha  - 1} \right)\sigma {{\left| k \right|}_\eta }}}} \right)\left( {\sum\limits_{{{\left| k \right|}_\eta } > K} {|\hat f\left( k \right)|{{\rm e}^{\sigma {{\left| k \right|}_\eta }}}} } \right) \leqslant {{\rm e}^{ - \left( {1 - \alpha } \right)\sigma K}}{\left\| {{\mathcal{R}_K}f} \right\|_\sigma }.
\end{align*}
\end{proof}

\begin{lemma}[Diophantine small divisor Lemma]\label{VECSM}
	Let $ \eta>0 $ and  $ \mu>1 $ be given. Then, for $ N \gg1 $ and a positive constant $ C(\eta,\mu) $ only depending on $ \eta $ and $ \mu $, we have that for all $ k \in \mathbb{Z}_ * ^\infty $:
	\[\mathop {\sup }\limits_{{{\left| k \right|}_\eta } < N} \prod\limits_{j \in \mathbb{Z}} {\left( {1 + {{\left\langle j \right\rangle }^\mu }{{\left| {{k_j}} \right|}^\mu }} \right)}  \leqslant \exp \left( {C\left( {\eta ,\mu } \right){N^{\frac{1}{{1 + \eta }}}}\ln \left( {1 + N} \right)} \right).\]
	Moreover, the order $ {N^{\frac{1}{{1 + \eta }}}}\ln \left( {1 + N} \right) $ is sharp.
\end{lemma}
\begin{proof}
The first part follows from \cite[Lemma B.2]{MR4201442}, and for the sake of completeness, we shall provide a brief proof. We only prove the case of the unilateral infinite sequence, with $ 2 \leqslant \mu=\eta\in \mathbb{N}^+ $, without loss of generality. For any fixed $ k \in \mathbb{Z}_ * ^\infty  $,  denote by $ m $ the number of nonzero components of $ k $. Then, if $ 1 \ll {\left| k \right|_\eta } \leqslant N $, we have
	\begin{align}
		N &\geqslant {\left| k \right|_\eta } = \sum\limits_{j \in \mathbb{N}} {\left| {{k_j}} \right|{{\left\langle j \right\rangle }^\eta }}  = \sum\limits_{i = 1}^m {\left| {{k_{{j_i}}}} \right|{{\left\langle {{j_i}} \right\rangle }^\eta }}  \geqslant \sum\limits_{i = 1}^m {{{\left\langle {{j_i}} \right\rangle }^\eta }} \notag \\
		\label{VECMULn}	& \geqslant \sum\limits_{i = 1}^m {{i^\eta }}  = {\mathcal{O}^\# }\left( {\int_1^m {{x^\eta }{\rm d}x} } \right) = {\mathcal{O}^\# }\left( m^{\frac{1}{{1 + \eta }}} \right),
	\end{align}
that is, $ m = \mathcal{O}\big( N^{\frac{1}{{1 + \eta }}} \big) $. Now, utilizing the inequality
	\[{\left| {{k_j}} \right|^\eta }{\left\langle j \right\rangle ^\eta } \leqslant {\left( {\left| {{k_j}} \right|{{\left\langle j \right\rangle }^\eta }} \right)^\eta } \leqslant {\left( {\sum\limits_{j \in \mathbb{N}} {\left| {{k_j}} \right|{{\left\langle j \right\rangle }^\eta }} } \right)^\eta } = \left| k \right|_\eta ^\eta  \leqslant {N^\eta },\]
we know that there exists a universal constant $ {{C \left(\eta\right) }}>0 $ only depending on $ \eta $, such that
	\begin{align}
		\mathop {\sup }\limits_{0 < {{\left| k \right|}_\eta } \leqslant N} \prod\limits_{j \in \mathbb{N}} {\left( {1 + {{\left| {{k_j}} \right|}^\eta }{{\left\langle j \right\rangle }^\eta }} \right)}  &= \mathop {\sup }\limits_{0 < {{\left| k \right|}_\eta } \leqslant N} \exp \left( {\sum\limits_{j \in \mathbb{N}} {\ln \left( {1 + {{\left| {{k_j}} \right|}^\eta }{{\left\langle j \right\rangle }^\eta }} \right)} } \right)\notag \\
		& \leqslant \exp \left( {\sum\limits_{ i  = 1}^m {\ln \left( {1 + {N^\eta }} \right)} } \right)\notag \\
		\label{VECMULXIAOCHUSHU}	& \leqslant \exp \left( {{{{C \left(\eta\right) }} }{N^{\frac{1}{{1 + \eta }}}}\ln \left(1+N\right)} \right).
	\end{align}
Next, we prove the second part. If we consider that the nonzero components of some $ \tilde k \in \mathbb{Z}_ * ^\infty $ are all modulus $ 1 $, and they are consecutively from $ 0 $, namely
\[\tilde k = (0, \ldots ,0,\underbrace {1,1, \ldots ,1}_m,0, \ldots ),\]
  then the argument in \eqref{VECMULn} leads to $ m = \mathcal{O}^{\#}\big( N^{\frac{1}{{1 + \eta }}} \big) $ as $  {\left| k \right|_\eta } \sim N $. Therefore, it follows that
	\begin{align*}
		\prod\limits_{j \in \mathbb{N}} {\left( {1 + {{| {{\tilde k_j}} |}^\eta }{{\left\langle j \right\rangle }^\eta }} \right)}  &= \exp \left( {\sum\limits_{j = 0}^m {\ln \left( {1 + {{\left\langle j \right\rangle }^\eta }} \right)} } \right) = \exp \left( {\eta \mathcal{O}^\#\left( {\int_1^m {\ln x{\rm d}x} } \right)} \right)\\
		& = \exp \left( {\eta \mathcal{O}^\#\left( {m\ln m} \right)} \right) = \exp \left( {\mathcal{O}^\#\left( {N^{\frac{1}{{1 + \eta }}}\ln \left(1+N\right)} \right)} \right),
	\end{align*}
	which shows the sharpness of the order $ {N^{\frac{1}{{1 + \eta }}}}\ln \left( {1 + N} \right) $.
\end{proof}

\subsection{Linearization of vector fields over $ \mathbb{T}^\infty $ via a non-Newtonian iteration}
In what follows, we always assume that the perturbation $ P $ of a constant vector field $ \omega $ on $\mathbb{T}^\infty$ has the Fourier  expansion series given by $ P = \sum\nolimits_{0 \ne k \in {\mathbb{Z}^\infty_*}} {{\hat P(k)}{{\rm e}^{{\rm i} {k\cdot x}  }}}  $, and possesses a  certain different regularity under different nonresonance for the frequency $ \omega $, which will be determined in Section \ref{VEC KAM Iteration}.

\subsubsection{Abstract  $ m $-weighted KAM Theorem}\label{VECABST}
To prove the specific KAM theorems in this paper, let us first establish an  abstract $ m $-weighted KAM result. As one will see later, our convergence rate can be \textit{arbitrarily slow} (which depends on the convergence rate of the  series $\sum\nolimits_{j=0}^\nu  {{\Delta _j}{\varepsilon _j}}  $ in the Iterative Lemma \ref{VECIterative}). In  KAM history, the iteration is always  Newtonian, that is,  super-exponential. However, when weakening regularity is considered, such a rapid convergence will become resistant. As demonstrated  in \cite{MR2350326,PoarXiv,TLCCM}, the convergence rate could also be arbitrarily slow, therefore sharp  regularity (in view of the historical counterexamples) has been achieved. Our Theorem \ref{VECABSTRACT} follows from the same viewpoint, and as a consequence, we obtain such a sharp result in the infinite-dimensional case for the first time.

\begin{theorem}[\textbf{Abstract $ m $-weighted KAM}]\label{VECABSTRACT}
	Let $ b>1 $ and a nonresonant constant vector field $ \omega \in  \mathbb{T}^\infty  $ be given.
	\\
	(I)	Suppose that there exist some $ 0<{q}<1 $ and a non-negative  sequence $ {\left\{ {{d_\nu }} \right\}_{\nu  \in \mathbb{N}}} $ with  $ \mathop {\overline {\lim } }\limits_{\nu  \to \infty } d_\nu ^{ - 1}{d_{\nu  + 1}} < b $, such that
	\begin{equation}\label{VECwudi}
	\sum\limits_{\nu  = 0}^\infty  {{q^\nu }\exp \left( { - \gamma \sum\limits_{j = 0 }^{\nu  - 1} {{d_j}} } \right){b^\nu }\mathop {\max }\limits_{0 < {{\left| k \right|}_\eta } \leqslant {b^\nu }} {{\left| {k \cdot \omega } \right|}^{ - 1}}}  <  + \infty ,
	\end{equation}
	where $ \gamma  = \frac{3}{4}\left( {1 - \lambda } \right) $ and $ \lambda  = {b^{ - 1}}\mathop {\overline {\lim } }\limits_{\nu  \to \infty } d_\nu ^{ - 1}{d_{\nu  + 1}} $.  Then there exists a   weight $ m $, as long as the perturbation $ P $ is sufficiently small in the sense that $ {\left\| P \right\|_m} \ll 1 $,	one has a modifying term $ \tilde \omega $ and a nearly identical diffeomorphism $ \Psi $ such that $ \omega -\tilde \omega + {P } $ is conjugated to $ \omega $, i.e.,
	\[	\Psi^ * \left( {\omega -\tilde \omega + {P } } \right) = \omega .\]
	(II) The weight $ m $ could be chosen as 
\[\mathop {\overline {\lim } }\limits_{\mu  \to \infty } \frac{{\rho \left( \mu  \right){{\rm e}^{{2^{ - 1}}\left( {1 + \lambda } \right)b{d_{\mu  - 1}}}}}}{{m\left( {{b^{\mu  - 1}}} \right)}} <  + \infty \]
	with
\[\rho \left( \mu  \right) = \sum\limits_{\nu  = \mu }^\infty  {{q^{\nu  - \mu }}\exp \left( { - \gamma \sum\limits_{j = \mu }^{\nu  - 1} {{d_j}} } \right){b^\nu }\mathop {\max }\limits_{0 < {{\left| k \right|}_\eta } \leqslant {b^\nu }} {{\left| {k \cdot \omega } \right|}^{ - 1}}} .\]
\end{theorem}

The proof of the Abstract  $ m $-weighted KAM Theorem \ref{VECABSTRACT} will be detailed  from Section \ref{VEC Main strategy} to Section \ref{VEC Uniform convergence}.

\subsubsection{Main strategy}\label{VEC Main strategy}
 Firstly, let us present the classical KAM strategy of our proof.  It should be noted that although we apply the idea of P\"oschel in his preprint \cite{PoarXiv}, it is indeed different in many details, for instance, in the construction of the iterative sequences (especially, we introduce a balancing sequence $ \{d_\nu\}_{\nu \in \mathbb{N}} $ to overcome the nonresonance, because P\"oschel's idea indeed only works for the Diophantine nonresonance in the finite-dimensional case) as well as the  infinite-dimensional structure. 

Assume that we have constructed  a modifying term $ \tilde\omega $ and a   coordinate transformation $ \Psi $ such that
\begin{equation}\label{VECshoulian-}
	{\Psi ^ * }\left( {\omega + P -\tilde\omega} \right) = \omega + Q,
\end{equation}
where $ P $ is a finitely order Fourier series of the original perturbation (namely at most of order $ {{{\left| k \right|}_\eta } > K} $ and without the constant term).  Then using the Banach contraction theorem (indeed, P\"oschel employed the Brouwer fixed point theorem in \cite{PoarXiv}, but we have to apply the Banach contraction theorem instead of it because the dimension is infinite) on a  ball $ \mathscr{B} $ via a special metric (norm), we obtain another modifying term $ \omega' $ and a coordinate transformation $ \Phi $ (note that they admit uniqueness) such that
\begin{equation}\label{Q^+}
	{\Phi ^ * }\left( {\omega + Q - {\Psi ^ * }\omega'} \right) = \omega + {Q^ + },
\end{equation}
here $ Q^+ $ is an intermediate quantity, and moreover we could obtain accurate estimates for $ \omega' $ and $ \Phi $ due to the quantitative property of $ \mathscr{B} $.   Now define the new modifying term   ${\tilde{\omega}_ + } = {\tilde{\omega} } + \omega'$ and the coordinate transformation $ {\Psi _ + }: = \Psi  \circ \Phi  $. Then it follows from \eqref{Q^+} that
\begin{align*}
	\Psi _ + ^ * \left( {\omega + P - {{\tilde{\omega}_ + } }} \right) &= {\Phi ^ * }\left( {{\Psi ^ * }\left( {\omega + P - {\tilde{\omega} }} \right) - {\Psi ^ * }\omega'} \right)\\
	& = {\Phi ^ * }\left( {\omega + Q - {\Psi ^ * }\omega'} \right)\\
	: &= \omega + {Q^ + },
\end{align*}
which leads to
\begin{align}
	\Psi _ + ^ * \left( {\omega + {P_ + } - {\tilde{\omega}_ + }} \right) &= \omega + {Q^ + } + \Psi _ + ^ * \left( {{P_ + } - P} \right)\notag \\
	\label{VECshoulian}: &= \omega + {Q_ + },
\end{align}
where the order of $ P_+ $ is larger than $ P $ and tends to infinite order during the iteration process. Recall \eqref{VECshoulian-}, therefore \eqref{VECshoulian} provides one cycle of the KAM iteration. Finally, we will prove the uniform convergence of \eqref{VECshoulian}, which implies  the desired conjugacy as
\[ {\Psi _\infty^* }\left( {\omega + {P_\infty } - \tilde{\omega}}_\infty \right) = \omega + {Q_\infty } = \omega.\]
Here we measure the original perturbation $ P_\infty $ (note that $ P_+ $ tends to $ P_\infty $ because it is indeed a truncation of $ P_\infty $) by a special $ m $-norm instead of the usual analytic norm, and therefore the regularity of it might be Gevrey regularity or  even $C^\infty$ regularity.
\subsubsection{KAM Step}
In this section, we establish the Step Lemma \ref{VECSTEP} that will serve as a crucial induction step in the KAM iteration.
\begin{lemma}[Step Lemma]\label{VECSTEP}
Let $ \lambda\in (0,1) $ and $ \alpha  := {2^{ - 1}}\left( {{\lambda} + 1} \right) \in \left( {0,1} \right) $ be given. Consider $ {\Psi ^ * }\left( {\omega  + P} \right) = \omega  + Q $.  Assume that 
\begin{equation}\label{VECSMALL}
	4\Delta {\left\| Q \right\|_\sigma } \leqslant \kappa := \frac{1}{4} \wedge \left( {\frac{1}{\alpha } - 1} \right),\;\;{\left\| {D\Psi  - \mathbb{I}} \right\|_\sigma } \leqslant \frac{1}{7}
\end{equation}
with $ \Delta : = K{\max _{{{\left| k \right|}_\eta } \leqslant K}}{\left| {k \cdot \omega } \right|^{ - 1}} $ and $ \sigma K \geqslant {\left( {1 - {\lambda}} \right)^{ - 1}} $. Then one can find a (unique) modifying term $ {\omega '} $ and a  (unique) transformation $ \Phi  $ satisfying
\begin{equation}\label{VECKONGJ}
	\Delta {\left| {\omega '} \right|} \vee K\left\| {\Phi  - {\rm id}} \right\|_{\alpha\sigma} \leqslant 4\Delta {\left\| Q \right\|_\sigma } \leqslant \kappa,
\end{equation}
 such that 
\begin{equation}\label{VECSTEPGONGE}
	{\Phi ^ * }{\Psi ^ * }\left( {\omega  + P - \omega '} \right) = \omega  + {Q^ + }.
\end{equation}
Moreover,  it holds that
\begin{equation}\label{VECSTEPQ+}
	{\left\| {{Q^ + }} \right\|_{\lambda\sigma }} \leqslant 12{{\rm e}^{ - \gamma\sigma K}}{\left\| Q \right\|_\sigma },\;\;\gamma : = \frac{3}{4}\left( {1 - \lambda } \right) > 0.
\end{equation}
\end{lemma}
\begin{proof}
Recall the conjugacy in  \eqref{Q^+}, that is,
\[{\Phi ^ * }\left( {\omega + Q - {\Psi ^ * }\omega'} \right) = \omega + {Q^ + }.\]
Then, letting $ \Phi ={\rm id}+\hat \Phi $ yields that
\begin{align}
	D\hat \Phi  \cdot \omega + D\Phi  \cdot {Q^ + } &= D\hat \Phi  \cdot \left( {\omega + {Q^ + }} \right) - \omega\notag\\
	&  = \left( {\omega + Q - {\Psi ^ * }\omega'} \right) \circ \Phi  - \omega\notag \\
	\label{VECzuhe}& = \left( {Q - {\Psi ^ * }\omega'} \right) \circ \Phi .
\end{align}
 We solve the following equations
\begin{equation}\label{VECzuhe2}
	\left\{ \begin{gathered}
		D\hat \Phi  \cdot \omega = {\mathcal{T}_K}\left( {Q - {\Psi ^ * }\omega'} \right) \circ \Phi,  \hfill \\
		D\Phi  \cdot {Q^{+}} = {\mathcal{R}_{K}}\left( {Q - {\Psi ^ * }\omega'} \right) \circ \Phi  \hfill \\ 
	\end{gathered}  \right.
\end{equation}
instead of solving \eqref{VECzuhe}.  Consider the first equation in \eqref{VECzuhe2}. With 
\begin{equation}\label{VECTheta}
	\Theta : = D{\Psi ^{ - 1}}\left( {D\Psi  - \mathbb{I}} \right)
\end{equation}
and the map $ \mathscr{T} $ defined by
\begin{equation}\label{VECHUATDY}
	\mathscr{T}\big( {\omega',\hat \Phi } \big): = \left( {Q + \Theta \omega'} \right) \circ \big( {{\rm id} + \hat \Phi } \big),
\end{equation}
one calculates the equivalent form of the first equation in \eqref{VECzuhe2} as
\begin{align}
	D\hat \Phi  \cdot \omega + \omega' &= {\mathcal{T}_K}\left( {Q - {\Psi ^ * }\omega'} \right) \circ \Phi  + \omega'\notag\\
	& = {\mathcal{T}_K}\left( {Q - \omega' + \Theta \omega'} \right) \circ \Phi  +\omega'\notag\\
	& = {\mathcal{T}_K}\left( {\left( {Q + \Theta \omega'} \right) \circ \Phi  - \omega'} \right) + \omega'\notag\\
	& = {\mathcal{T}_K}\left( {Q + \Theta \omega'} \right) \circ \big( {{\rm id} + \hat \Phi } \big)\notag\\
& = {\mathcal{T}_K}\mathscr{T}\big( {\omega',\hat \Phi } \big).\notag
\end{align}
Note that $ \omega' $ is constant, we therefore solve the following equations:
\begin{equation}\label{VECzuhe4}
	\left\{ \begin{gathered}
		D\hat \Phi  \cdot \omega = \left( {{\mathcal{T}_k} - {\mathcal{T}_0}} \right)\mathscr{T}\big( {\omega',\hat \Phi } \big), \hfill \\
		\omega' = {\mathcal{T}_0}\mathscr{T}\big( {\omega',\hat \Phi } \big). \hfill \\ 
	\end{gathered}  \right.
\end{equation}
Since the operator $ D \cdot \omega = \sum\nolimits_{i = 1}^n {{\omega _i}{\partial _{{x_i}}}} =\omega \cdot \partial $ yields the small divisors and there is no constant term in the first equation in \eqref{VECzuhe4}, then by the Small divisor  Lemma \ref{VECL4}, we have
\begin{equation}\label{VECzuhe5}
	\left\{ \begin{gathered}
		\omega' = {\mathcal{T}_0}\mathscr{T}\big( {\omega',\hat \Phi } \big), \hfill \\
		\hat \Phi  = \mathscr{L}\left( {{\mathcal{T}_k} - {\mathcal{T}_0}} \right)\mathscr{T}\big( {\omega',\hat \Phi } \big). \hfill \\ 
	\end{gathered}  \right.
\end{equation}
With the fact that $   {\left\| {X } \right\|_\sigma }={\left\| {X } \right\|_0}=\left|{X }\right| $ holds for constant $ X $ in mind,  the solution to  \eqref{VECzuhe5} is equivalent to the fixed point of the map $ \mathscr{Q} $
\begin{equation}\label{VECHUAQDY}
	\mathscr{Q}{\big( {\omega',\hat \Phi } \big)}:=\left\{ \begin{gathered}
		{\omega'_1} = {\mathcal{T}_0}\mathscr{T}\big( {\omega',\hat \Phi } \big), \hfill \\
		{{\hat \Phi }_1} = \mathscr{L}\left( {{\mathcal{T}_k} - {\mathcal{T}_0}} \right)\mathscr{T}\big( {\omega',\hat \Phi } \big) \hfill \\ 
	\end{gathered}  \right.
\end{equation}
in a ball $ \mathscr{B} $ defined as 
\begin{equation}\label{VECBALL}
	\mathscr{B}: = \left\{ {\big( {\omega',\hat \Phi } \big) \in \Omega :\Delta {\left| {\omega '} \right|} \vee K{\big\| {\hat \Phi } \big\|_{\alpha\sigma} } \leqslant 4\Delta {\left\| Q \right\|_\sigma } \leqslant \kappa} \right\},
\end{equation}
where the space $ \Omega $ is chosen as
\[\Omega : = \left\{ {\big( {\omega',\hat \Phi } \big):\omega' = \text{constant},\;\hat \Phi  = \sum\limits_{0 < \left| k \right|_\eta \leqslant K} {{{\hat \Phi }_k}{{\rm e}^{{\rm i}{k\cdot x}  }}}  } \right\}.\]
This motivates us to apply the Banach contraction theorem, as we can  verify that $ \mathscr{B} $ is indeed a \textit{Banach ball} endowed with the metric (norm) in \eqref{VECBALL}. To achieve this, we need to prove that for every $ \big( {{\omega'_1},{{\hat \Phi }_1}} \big) $ arising from the initial value $ {\big( {\omega',\hat \Phi } \big)} $ under the map $ \mathscr{Q} $ defined in \eqref{VECHUAQDY}, it still belongs to the ball $ \mathscr{B} $, in the sense of the given metric (norm) in \eqref{VECBALL}. 

With \eqref{VECTheta}, the Neumann Lemma \ref{VECL1} and the smallness assumption \eqref{VECSMALL}, we have
\begin{equation}\label{VECTHETA}
	{\left\| \Theta  \right\|_\sigma } = {\left\| {D{\Psi ^{ - 1}}\left( {D\Psi  - \mathbb{I}} \right)} \right\|_\sigma } \leqslant {\left\| {D{\Psi ^{ - 1}}} \right\|_\sigma }{\left\| {D\Psi  - \mathbb{I}} \right\|_\sigma } \leqslant \frac{{1/7}}{{1 - 1/7}} = \frac{1}{6},
\end{equation}
which yields 
\begin{align}
	{\left\| {Q + \Theta \omega '} \right\|_\sigma } &\leqslant {\left\| Q \right\|_\sigma } + {\left\| {\Theta \omega '} \right\|_\sigma } \leqslant {\left\| Q \right\|_\sigma } + {\left\| \Theta  \right\|_\sigma }{\left| {\omega '} \right|}\notag \\
\label{VECQ+thetaomega'}	& \leqslant {\left\| Q \right\|_\sigma } + \frac{1}{6} \cdot 4{\left\| Q \right\|_\sigma } \leqslant 2{\left\| Q \right\|_\sigma },
\end{align}
because $ \omega' $ is constant. In view of $ \sigma K \geqslant {\left( {1 - {b^{ - 1}}} \right)^{ - 1}} $, one then concludes from \eqref{VECBALL} that
\begin{equation}\label{VECHATFAI}
	{\big\| {\hat \Phi } \big\|_{\alpha \sigma }} \leqslant \frac{1}{{4K}} \leqslant \frac{{1 - \lambda}}{4}\sigma  = \frac{{1 - \alpha }}{2}\sigma .
\end{equation}
Therefore, it follows from \eqref{VECHUATDY}, \eqref{VECQ+thetaomega'},  \eqref{VECHATFAI} and the Transformation Lemma \ref{VECL2} that 
\begin{align}
{\left\| {\mathscr{T}\big( {\omega ',\hat \Phi } \big)} \right\|_{\alpha \sigma }} &= {\left\| {\left( {Q + \Theta \omega '} \right) \circ \big( {{\rm id} + \hat \Phi } \big)} \right\|_{\alpha \sigma }} \leqslant {\left\| {Q + \Theta \omega '} \right\|_{\alpha \sigma  + {2^{ - 1}}\left( {1 - \alpha } \right)\sigma }}\notag \\
\label{VECT}& = {\left\| {Q + \Theta \omega '} \right\|_{{2^{ - 1}}\left( {1 + \alpha } \right)\sigma }} \leqslant {\left\| {Q + \Theta \omega '} \right\|_\sigma } \leqslant 2{\left\| Q \right\|_\sigma }.
\end{align}
On these grounds, one can observe from \eqref{VECHUAQDY} and \eqref{VECT} that
\[\Delta {\left| {{{\omega '_1}}} \right|} = \Delta {\left\| {{{\omega '_1}}} \right\|_{\alpha \sigma }} = \Delta {\left\| {{\mathcal{T}_0}\mathscr{T}\big( {\omega ',\hat \Phi } \big)} \right\|_{\alpha \sigma }} \leqslant \Delta {\left\| {\mathscr{T}\big( {\omega ',\hat \Phi } \big)} \right\|_{\alpha \sigma }} \leqslant 4\Delta {\left\| Q \right\|_\sigma },\]
and
\begin{align}
K{\big\| {{{\hat \Phi }_1}} \big\|_{\alpha \sigma }} &= K{\left\| {\mathscr{L}\left( {{\mathcal{T}_k} - {\mathcal{T}_0}} \right)\mathscr{T}\big( {\omega ',\hat \Phi } \big)} \right\|_{\alpha \sigma }}\notag\\ &\leqslant K\mathop {\max }\limits_{0 < {{\left| k \right|}_\eta } \leqslant K} {\left| {k \cdot \omega } \right|^{ - 1}}{\left\| {\left( {{\mathcal{T}_k} - {\mathcal{T}_0}} \right)\mathscr{T}\big( {\omega ',\hat \Phi } \big)} \right\|_{\alpha \sigma }}\notag \\
\label{VECXIANXING}& \leqslant \Delta {\left\| {\mathscr{T}\big( {\omega ',\hat \Phi } \big)} \right\|_{\alpha \sigma }} \leqslant 4\Delta {\left\| Q \right\|_\sigma }.
\end{align}
These imply that the map $ \mathscr{Q} $ is indeed a self map in the Banach ball $ \mathscr{B} $. We are now in a position to prove its contraction property. 

Utilizing the definition of $ 	\mathscr{T} $ in  \eqref{VECHUATDY}, one obtains
\begin{align}
	\mathscr{T}(\omega',\hat \Phi ) - \mathscr{T}(\omega'',\hat \Phi ') &= \left( {Q + \Theta \omega'} \right) \circ \big({\rm id} + \hat \Phi \big) - \left( {Q + \Theta \omega''} \right) \circ \big({\rm id} + \hat \Phi '\big)\notag \\
	& = \left[ {\left( {Q + \Theta \omega'} \right) \circ \Phi  - \left( {Q + \Theta \omega'} \right) \circ \Phi '} \right] + \left[ {\left( {\Theta \left( {\omega' - \omega''} \right)} \right) \circ \Phi' } \right]\notag \\
	\label{VECJ1J2}:&= {\mathcal{J}_1} + {\mathcal{J}_2}.
\end{align}
By recalling that $ \hat \Phi $ is a truncated Fourier series, it follows from \eqref{VECBALL} and the Cauchy type Lemma \ref{VECCAULE} that
\[{\| {D\hat \Phi } \|_{\alpha \sigma }} \leqslant K{\big\| {\hat \Phi } \big\|_{\alpha \sigma }} \leqslant {\alpha ^{ - 1}-1},\]
then similarly by the Neumann Lemma \ref{VECL1}, one has
\begin{equation}\label{VECDFAI}
{\left\| {D\Phi } \right\|_{\alpha \sigma }} \vee {\left\| {D{\Phi ^{ - 1}}} \right\|_{\alpha \sigma }} \leqslant  {\alpha ^{ - 1}}.
\end{equation}
Note \eqref{VECDFAI} implies that any map in $ \mathscr{B} $ maps $ \mathbb{T}^\infty_{\alpha \sigma} $ into $ \mathbb{T}^\infty_{\sigma} $. Then by the Cauchy's estimate in Lemma \ref{VECCAULE},  we get
\begin{align}
{\left\| {{\mathcal{J}_1}} \right\|_{\alpha \sigma }} &\leqslant \frac{1}{{{\rm e}\left( {1 - \alpha } \right)\sigma }}{\left\| {Q + \Theta \omega '} \right\|_\sigma }{\left\| {\Phi  - \Phi '} \right\|_{\alpha \sigma }}\notag \\
& \leqslant \frac{1}{{{\rm e}\left( {1 - \alpha } \right)\sigma K}} \cdot 2{\left\| Q \right\|_\sigma }K{\left\| {\Phi  - \Phi '} \right\|_{\alpha \sigma }}\notag \\
	\label{VECJ1}&  \leqslant 4{\left\| Q \right\|_\sigma }K{\left\| {\Phi  - \Phi '} \right\|_{\alpha \sigma }},
\end{align}
where $ \sigma K \geqslant {\left( {1 - \lambda} \right)^{ - 1}} $ is used in \eqref{VECJ1}. Besides, one observes that \eqref{VECHATFAI} also holds for $ \hat \Phi' $, then the Transformation Lemma \ref{VECL2} gives (similar to \eqref{VECT})
\begin{align}
	{\left\| {{\mathcal{J}_2}} \right\|_{\alpha \sigma}} &= {\left\| {\left( {\Theta \left( {\omega' - \omega''} \right)} \right) \circ \Phi '} \right\|_{\alpha \sigma}} \leqslant {\left\| {\Theta \left( {\omega' - \omega''} \right)} \right\|_{\sigma}}\notag \\
\label{VECJ2}& \leqslant {\left\| \Theta  \right\|_\sigma }{\left\| {\omega ' - \omega ''} \right\|_\sigma } \leqslant \frac{1}{4}{\left| {\omega ' - \omega ''} \right|},
\end{align}
because $ \omega'-\omega'' $ is constant, where \eqref{VECTHETA} is employed in \eqref{VECJ2}. Now, substituting \eqref{VECJ1} and \eqref{VECJ2} into \eqref{VECJ1J2} yields 
\begin{align}
	\label{VECFAICHA'}K{\big\| {{{\hat \Phi }_1} - {{\hat \Phi }_1}' } \big\|_{\alpha \sigma }} &\leqslant \Delta {\left\| {\mathscr{T}\big( {\omega ',\hat \Phi } \big) - \mathscr{T}\big( {\omega '',\hat \Phi ''} \big)} \right\|_{\alpha \sigma }} \\
	&  \leqslant \Delta {\left\| {{\mathcal{J}_1}} \right\|_{\alpha \sigma }} + \Delta {\left\| {{\mathcal{J}_2}} \right\|_{\alpha \sigma }} \\
	&  \leqslant 4\Delta {\left\| Q \right\|_\sigma }K{\big\| {\hat \Phi  - \hat \Phi '} \big\|_{\alpha \sigma }} + \frac{1}{4}\Delta \left| {\omega ' - \omega ''} \right|\notag \\
	\label{VECFAICHA}&  \leqslant \frac{1}{4}\left( {K{{\big\| {\hat \Phi  - \hat \Phi '} \big\|}_{\alpha \sigma }} + \Delta \left| {\omega ' - \omega ''} \right|} \right)
\end{align}
due to \eqref{VECXIANXING} and  linearity, and one uses $ 4\Delta {\left\| Q \right\|_s} \leqslant 1/4 $ from the definition of the Banach  ball $ \mathscr{B} $ in \eqref{VECBALL}. Similarly, from  \eqref{VECFAICHA'} and \eqref{VECFAICHA}, we obtain that
\begin{align}
\Delta \left| {{\omega' _1}  - {\omega'' _1}} \right| &= \Delta {\left\| {{\mathcal{T}_0}\left( {\mathscr{T}\big( {\omega ',\hat \Phi } \big) - \mathscr{T}\big( {\omega '',\hat \Phi ''} \big)} \right)} \right\|_0}\notag \\
	&   \leqslant \Delta {\left\| {\mathscr{T}\big( {\omega ',\hat \Phi } \big) - \mathscr{T}\big( {\omega '',\hat \Phi ''} \big)} \right\|_{\alpha \sigma }}\notag \\
	\label{VECZCHA}&  \leqslant \frac{1}{4}\left( {K{{\big\| {\hat \Phi  - \hat \Phi '} \big\|}_{\alpha \sigma }} + \Delta \left| {\omega ' - \omega ''} \right|} \right),
\end{align}
because $ {{\omega' _1}  - {\omega'' _1}} $ is also constant. Combining \eqref{VECFAICHA} and \eqref{VECZCHA}, one concludes that the map $ \mathscr{Q} $ is indeed a $ 1/2 $-contraction in the Banach ball $ \mathscr{B} $ with respect to the metric (norm)
\[{K{{\big\| {\hat \Phi  - \hat \Phi '} \big\|}_{\alpha \sigma }} \vee \Delta \left| {\omega ' - \omega ''} \right|}.\]
In this case, applying the Banach contraction theorem, one obtains a unique fixed point $ \big( {\omega ',\hat \Phi } \big) \in \mathscr{B} $ of the map $ \mathscr{Q} $, which corresponds to the solution to the first equation in \eqref{VECzuhe2} as we have discussed. Moreover, the Banach ball $ \mathscr{B} $ provides the estimate for $ \big( {\omega ',\hat \Phi } \big) \in \mathscr{B} $ as
\begin{equation}\label{VECbudongdian}
	\Delta {\left| {\omega '} \right|} \vee K{\big\| {\hat \Phi } \big\|_{\alpha\sigma} } \leqslant 4\Delta {\left\| Q \right\|_\sigma } \leqslant \kappa.
\end{equation}

Finally, note that $ Q^+ $ is determined by the second equation in \eqref{VECzuhe2}, then
\begin{equation}\label{VECQZHENG}
	{Q^{+}} = D{\Phi ^{ - 1}} \cdot {\mathcal{R}_{K}}\left( {Q - {\Psi ^ * }\omega'} \right) \circ \Phi .
\end{equation}
Recall $ {\left| {D\Psi  - \mathbb{I}} \right|_\sigma} \leqslant 1/7 $. It follows from the Neumann Lemma \ref{VECL1} and \eqref{VECbudongdian} that
\begin{align}
	{\left\| {{\Psi ^ * }\omega'} \right\|_\sigma} &= {\left\| {D{\Psi ^{ - 1}} \cdot \omega' \circ \Psi } \right\|_\sigma} = {\left\| {D{\Psi ^{ - 1}} \cdot Z} \right\|_\sigma}\notag \\
	& \leqslant {\left\| {D{\Psi ^{ - 1}}} \right\|_\sigma}{\left\| \omega' \right\|_\sigma} = {\left\| {D{\Psi ^{ - 1}}} \right\|_\sigma}{\left| \omega' \right|}\notag\\
	\label{PSIZ}& \leqslant \frac{7}{6}{\left| \omega' \right|} \leqslant \frac{7}{6} \cdot 4{\left\| Q \right\|_\sigma} = \frac{{14}}{3}{\left\| Q \right\|_\sigma},
\end{align}
because $ \omega' $ is constant. Therefore, by utilizing \eqref{VECQZHENG}, we prove that 
\begin{align}
{\left\| {{Q^ + }} \right\|_{{\lambda\sigma }}} &\leqslant {\left\| {D{\Phi ^{ - 1}}} \right\|_{\alpha \sigma}}{\left\| {{\mathcal{R}_{K }}\left( {\left( {Q - {\Psi ^ * }\omega'} \right) \circ \Phi } \right)} \right\|_{{\lambda\sigma }}}\notag\\
	\label{VECQ+2}&\leqslant {\alpha}^{-1} {\left\| {{\mathcal{R}_{K }}\left( {Q - {\Psi ^ * }\omega'} \right)} \right\|_{{\lambda}\sigma  + {2^{ - 1}}\left( {1 - \alpha } \right)\sigma }}\\
&= {\alpha}^{-1} {\left\| {{\mathcal{R}_{K }}\left( {Q - {\Psi ^ * }\omega'} \right)} \right\|_{{ {4} ^{ - 1}}\left( {1+3\lambda} \right)\sigma}} \notag \\	
	\label{VECQ+3}&  \leqslant {\alpha ^{ - 1}}{{\rm e}^{ - \left( 1-{ {4} ^{ - 1}}\left( {1+3\lambda} \right)\sigma \right)\sigma K}}{\left\| {Q - {\Psi ^ * }\omega '} \right\|_\sigma }\\
\label{VECQ+15}& = 2{\left( {1+\lambda} \right)^{ - 1}}{{\rm e}^{ - \frac{3}{4}\left( {1 - \lambda } \right)\sigma K}}{\left\| {Q - {\Psi ^ * }\omega '} \right\|_\sigma },
\end{align}
 where \eqref{VECDFAI} and the  Transformation Lemma \ref{VECL2} are employed in \eqref{VECQ+2} (similar to that in \eqref{VECT}), \eqref{VECQ+3} is due to the Residual term Lemma \ref{VECL5} (with $ \alpha  = {4^{ - 1}}\left( {1 + 3\lambda } \right) \in \left( {0,1} \right) $). Therefore, by observing 
 \[{\left\| {Q - {\Psi ^ * }\omega '} \right\|_\sigma } \leqslant {\left\| Q \right\|_\sigma } + {\left\| {{\Psi ^ * }\omega '} \right\|_\sigma } \leqslant {\left\| Q \right\|_\sigma } + \frac{{14}}{3}{\left\| Q \right\|_\sigma } = \frac{{17}}{3}{\left\| Q \right\|_\sigma }\]
 through \eqref{PSIZ}, we conclude the desired estimate from \eqref{VECQ+15} that
 \[{\left\| {{Q^ + }} \right\|_{{\lambda}\sigma }} \leqslant \frac{{34}}{{3\left( {1 + \lambda } \right)}}{{\rm e}^{ - \frac{3}{4}\left( {1 - \lambda } \right)\sigma K}}{\left\| Q \right\|_\sigma } \leqslant 12{{\rm e}^{ - \gamma \sigma K}}{\left\| Q \right\|_\sigma },\]
 provided with $ \gamma : = \frac{3}{4}\left( {1 - \lambda } \right) > 0 $. This completes the proof of the Step Lemma \ref{VECSTEP}.
\end{proof}

\subsubsection{KAM Iteration}\label{VEC KAM Iteration}
Let $ r\gg 1 $ and $ b >1 $ be given. For a non-negative sequence $ {\left\{ {{d_\nu }} \right\}_{\nu  \in {\mathbb{N}^ + }}} $ satisfying $ \mathop {\overline {\lim } }\nolimits_{\nu  \to \infty } d_\nu ^{ - 1}{d_{\nu  + 1}} < b $, let us define
\begin{equation}\label{VECxulie}
	{K_\nu } = {b^\nu },\;\;{\sigma_\nu } = {b^{ - \nu }}\left( {r + {d_\nu }} \right),\;\;{\Delta _\nu } = {K_\nu }\mathop {\max }\limits_{{0<{\left| k \right|}_\eta } \leqslant {K_\nu }} {\left| {k \cdot \omega } \right|^{ - 1}},\;\; \nu \in \mathbb{N}.
\end{equation}
Then it follows that $ \sigma_\nu K_\nu \geqslant r $, and we have $ \sigma_{\nu+1} \leqslant \lambda \sigma_\nu $ with $ \lambda : = {b^{ - 1}}\mathop {\overline {\lim } }\nolimits_{\nu \to \infty } d_\nu ^{ - 1}{d_{\nu  + 1}} \in \left( {0,1} \right) $ by the assumption on $ d_\nu $ without loss of generality. As we will see later, such a \textit{balancing sequence} $ {\left\{ {{d_\nu }} \right\}_{\nu  \in {\mathbb{N}^ + }}} $ will overcome the nonresonance beyond  polynomial's type, for instance, the infinite-dimensional Diophantine nonresonance   in Definition \ref{VECINDIO}, and it is indeed \textit{indispensable}. 
Furthermore, the Cauchy type Lemma \ref{VECCAULE} and the Small divisor Lemma \ref{VECL4}  provide:
\begin{align}
	K_{\nu}{\left\| {\mathscr{L}\mathcal{T}_{K_\nu} f} \right\|_{\sigma_\nu}} \vee {\left\|	D {\mathscr{L}\mathcal{T}_{K_\nu} f} \right\|_{\sigma_\nu}} \leqslant \Delta_\nu {\left\| \mathcal{T}_{K_\nu}f \right\|_{\sigma_\nu}},\notag 
\end{align}
and
\[{\left\| \mathcal{R}_{K_\nu}f \right\|_{\alpha \sigma_\nu }} \leqslant {{\rm e}^{ - \left( {1 - \alpha } \right)r}} \cdot {{\rm e}^{ - \left( {1 - \alpha } \right){d_\nu }}}{\left\| \mathcal{R}_{K_\nu}f \right\|_{\sigma_\nu} }.\]
On these grounds, one could obtain the Step Lemma \ref{VECSTEP} under the above setting (namely \eqref{VECxulie}), whenever $ r\gg1 $. Now, let us introduce a non-negative  weight $m = {\left\{ {{m_k}} \right\}_{k \in {\mathbb{Z}_*^\infty}}} $, and set 
\[{\mathscr{K}_\nu }: = \left\{ {0 \ne k \in {\mathbb{Z}^n}:{K_{\nu  - 1}} < \left| k \right|_\eta \leqslant {K_\nu }} \right\}\]
for all $ \nu \in \mathbb{N} $. One will see later that the weight $ m $ could be appropriately and explicitly chosen from a bounded series function.  For the given perturbation $ P $ with Fourier expansion
\[P = \sum\limits_{0 \ne k \in {\mathbb{Z}_*^\infty}} {{\hat P(k)}{{\rm e}^{{\rm i} {k\cdot x}  }}}, \]
we shall use the $ K_\nu $-truncated term defined as 
\begin{equation}\notag 
	{P_\nu }: = \sum\limits_{0 < \left| k \right|_\eta \leqslant {K_\nu }} {{\hat P(k)}{{\rm e}^{{\rm i} {k\cdot x}  }}} ,\;\;\nu  \geqslant 0
\end{equation}
to approximate the original perturbation $ P $. In particular, we have $ P_\infty=P $. Moreover, we define $ {\Delta ^ * }{P_0}: = {P_0} $ and the difference term $ {\Delta ^ * }{P_\nu }: = {P_\nu } - {P_{\nu  - 1}} $ for $ \nu \geqslant 1 $. Then the above notations yield
\[{\Delta ^ * }{P_\nu } = \sum\limits_{k \in {\mathscr{K}_\nu }} {{\hat P(k)}{{\rm e}^{{\rm i} {k\cdot x}  }}} .\]

Now we are in a position to establish the  Iterative Lemma \ref{VECIterative}.

\begin{lemma}[Iterative Lemma]\label{VECIterative}
	Suppose that there exist some $ 0<{q}<1 $ and a non-negative  sequence $ {\left\{ {{d_\nu }} \right\}_{\nu  \in \mathbb{N}}} $ in \eqref{VECxulie} with  $ \mathop {\overline {\lim } }\nolimits_{\nu  \to \infty } d_\nu ^{ - 1}{d_{\nu  + 1}} < b $, such that
	\begin{equation}\label{VECAAA}
		\sum\limits_{\nu  = 0}^\infty  {{{ q}^\nu }\exp \left( { - \gamma\sum\limits_{j = 0}^{\nu  - 1} {{d_j}}  } \right){{\Delta _{\nu}} }}  <  + \infty ,
	\end{equation}
where $ \gamma  = \frac{3}{4}\left( {1 - \lambda } \right)$ and $ \lambda  = {b^{ - 1}}\mathop {\overline {\lim } }\nolimits_{\nu  \to \infty } d_\nu ^{ - 1}{d_{\nu  + 1}} $.	Then there exists a   weight $ m $ 	such that the followings hold, whenever the perturbation $ P $ of $ \omega $ is sufficiently small in the sense that $ \varepsilon  = {\left\| P \right\|_m} \ll 1 $.  For each $ P_\nu $ there exist a modifying term $\tilde \omega_\nu $ and a transformation $ \Psi_\nu $, such that
	\begin{equation}\label{VECbijin}
		\Psi _\nu ^ * \left( {\omega + {P_\nu } - {\tilde{\omega}_\nu }} \right) = \omega + {Q_\nu }
	\end{equation}
	with the estimates
	\begin{equation}\label{VECdiedai1}
	{\left\| {{Q_\nu }} \right\|_{{\sigma _\nu }}} \leqslant {\varepsilon _\nu }: = C\sum\limits_{\mu  = 0}^\nu  {{q^{\nu  - \mu }}\exp \left( { - \gamma \sum\limits_{j = \mu }^{\nu  - 1} {{d_j}} } \right)\frac{{{{\rm e}^{{2^{ - 1}}\left( {1 + \lambda } \right)b{d_{\mu  - 1}}}}}}{{m\left( {{b^{\mu  - 1}}} \right)}}{{\left\| {{\Delta ^ * }{P_\mu }} \right\|}_m}} ,
	\end{equation}
provided with some universal constant $ C>0 $, 	and
	\begin{equation}\label{VECdiedai2}
		{\| {D{{ \Psi }_\nu }-\mathbb{I}} \|_{{\sigma_\nu }}} \leqslant {\delta _\nu }: = \prod\limits_{\mu  = 0}^{\nu  - 1} {\left( {1 + 4{\Delta _\mu }{\varepsilon _\mu }} \right)}-1 .
	\end{equation}
  Moreover, one has 
	\begin{equation}\label{VECdiedai3}
		\left| {{{\tilde \omega }_{\nu  + 1}} - {{\tilde \omega }_\nu }} \right| \leqslant 4{\varepsilon _\nu },
	\end{equation}
	and 
	\begin{equation}\label{VECdiedai4}
		{\left\| {D{\Psi _{\nu  + 1}} - D{\Psi _\nu }} \right\|_{0}} \leqslant 12{\Delta _\nu }{\varepsilon _\nu }.
	\end{equation}
\end{lemma}
\begin{remark}\label{VECRHOMU}
Indeed, one could choose the weight $ m $ for which the boundedness  holds
\[\mathop {\overline {\lim } }\limits_{\mu  \to \infty } \frac{{\rho \left( \mu  \right){{\rm e}^{{2^{ - 1}}\left( {1 + \lambda } \right)b{d_{\mu  - 1}}}}}}{{m\left( {{b^{\mu  - 1}}} \right)}} <  + \infty ,\;\;\rho \left( \mu  \right): = \sum\limits_{\nu  = \mu }^\infty  {{q^{\nu  - \mu }}\exp \left( { - \gamma \sum\limits_{j = \mu }^{\nu  - 1} {{d_j}} } \right){b^\nu }\mathop {\max }\limits_{0 < {{\left| k \right|}_\eta } \leqslant {b^\nu }} {{\left| {k \cdot \omega } \right|}^{ - 1}}}. \]
\end{remark}
\begin{proof}
 For $ \nu=0 $, let us set $ Y_0=0 $, $ \Psi_0={\rm id} $. Consequently, $  \Psi_0-{\rm id}=0 $, $ Q_0 =P_0=\Delta ^ * {P_0 }$, and all estimates for $ \nu=0 $ are satisfied if one defines $ m\left( {{b^{-1}}} \right): = 1 $ and $ d_{-1}:=0 $, and the universal $ C>0 $ is chosen sufficiently large. Next, we will complete the proof of the Iterative Lemma \ref{VECIterative} by induction.
 
 It is evident that \eqref{VECAAA} implies  the function $ \rho(\mu) $ given in Remark \ref{VECRHOMU} is well defined, and the weight $ m $ could be suitably chosen from the boundedness condition (it should be noted that this is indeed achievable):
 \[\mathop {\overline {\lim } }\limits_{\mu  \to \infty } \frac{{\rho \left( \mu  \right){{\rm e}^{{2^{ - 1}}\left( {1 + \lambda } \right)b{d_{\mu  - 1}}}}}}{{m\left( {{b^{\mu  - 1}}} \right)}} <  + \infty.\]
  Next, we shall establish a crucial fact via the smallness assumption for the perturbation, namely $ \sum\nolimits_{\nu  = 0}^\infty  {{\Delta _\nu }{\varepsilon _\nu }}  \ll 1$. By exchanging the order of the summation, we get
 \begin{align*}
 \sum\limits_{\nu  = 0}^\infty  {{\Delta _\nu }{\varepsilon _\nu }}  &\leqslant C\sum\limits_{\nu  = 0}^\infty  {{b^\nu }\mathop {\max }\limits_{0 < {{\left| k \right|}_\eta } \leqslant {b^\nu }} {{\left| {k \cdot \omega } \right|}^{ - 1}}} \left( {\sum\limits_{\mu  = 0}^\nu  {{q^{\nu  - \mu }}\exp \left( { - \gamma \sum\limits_{j = \mu }^{\nu  - 1} {{d_j}} } \right)\frac{{{{\rm e}^{{2^{ - 1}}\left( {1 + \lambda } \right)b{d_{\mu  - 1}}}}{{\left\| {{\Delta ^ * }{P_\mu }} \right\|}_m}}}{{m\left( {{b^{\mu  - 1}}} \right)}}} } \right)\\
& = C\sum\limits_{\mu  = 0}^\infty  {\frac{{\rho \left( \mu  \right){{\rm e}^{{2^{ - 1}}\left( {1 + \lambda } \right)b{d_{\mu  - 1}}}}}}{{m\left( {{b^{\mu  - 1}}} \right)}}{{\left\| {{\Delta ^ * }{P_\mu }} \right\|}_m}}  = \mathcal{O}\left( {\sum\limits_{\mu  = 0}^\infty  {{{\left\| {{\Delta ^ * }{P_\mu }} \right\|}_m}} } \right) \\
&= \mathcal{O}\left( {{{\left\| P \right\|}_m}} \right) = \mathcal{O}\left( \varepsilon  \right) = o\left( 1 \right).
 \end{align*}
In this case, the Step Lemma \ref{VECSTEP} is valid for all $ \Psi_\nu $ and $ Q_\nu $ with $ \nu \in \mathbb{N} $, as long as $ \varepsilon>0 $ is sufficiently small.  Applying  the Step Lemma \ref{VECSTEP}, we obtain a modifying term $ \omega'_\nu $ and a transformation $ \Phi_\nu $ satisfying 
\begin{equation}\label{VECSTEP1}
	{\left| {\omega'_\nu} \right|} \leqslant 4{\left\| {{Q_\nu }} \right\|_{{\sigma_\nu }}} \leqslant 4{\varepsilon _\nu }
\end{equation}
and (also using the Cauchy type Lemma \ref{VECCAULE})
\begin{equation}\label{VECSTEP2}
	K{\left\| {{\Phi _\nu } - {\rm id}} \right\|_{\alpha {\sigma _\nu }}} \vee {\left\| {D{\Phi _\nu } - \mathbb{I}} \right\|_{\alpha {\sigma _\nu }}} \leqslant 4{\Delta _\nu }{\left\| {{Q_\nu }} \right\|_{{\sigma _\nu }}} \leqslant 4{\Delta _\nu }{\varepsilon _\nu }.
\end{equation}
By setting the new modifying term $ {{\tilde \omega }_{\nu  + 1}} = {{\tilde \omega }_\nu } + \omega'_\nu $ and the new transformation $ {\Psi _{\nu  + 1}} = {\Psi _\nu } \circ {\Phi _\nu } $, and recalling \eqref{VECshoulian}, we obtain
\begin{align}
	\Psi _{\nu  + 1}^ * \left( {\omega + {P_{\nu  + 1}} - {{\tilde \omega }_{\nu  + 1}}} \right) &= \omega + Q_\nu ^ +  + \Psi _{\nu  + 1}^ * {\Delta ^ * }{P_{\nu  + 1}}\notag \\
	\label{QNU+1}: &= \omega + {Q_{\nu  + 1}}.
\end{align}

With \eqref{QNU+1} in mind, we next establish the induction for $ {\left\| {{Q_\nu }} \right\|_{{\sigma _\nu }}} \leqslant {\varepsilon _\nu } $. Let $ r>0 $ be chosen sufficiently large such that $ 12{{\rm e}^{ - \gamma r}} \leqslant q \in \left( {0,1} \right) $. On the one hand, the Step Lemma \ref{VECSTEP} provides the estimate
\begin{equation}\label{VECQNU+}
	{\left\| {Q_\nu ^ + } \right\|_{\lambda {\sigma _\nu }}} \leqslant q {{\rm e}^{ - \gamma {d_\nu }}}{\left\| {{Q_\nu }} \right\|_{{\sigma _\nu }}}
\end{equation}
 with $ \lambda  = {b^{ - 1}}\mathop {\overline {\lim } }\nolimits_{\nu \to \infty } d_\nu ^{ - 1}{d_{\nu  + 1}} \in \left( {0,1} \right) $. On the other hand, one could derive the following similar to the previous arguments:
\begin{align}
{\left\| {\Psi _{\nu  + 1}^ * {\Delta ^ * }{P_{\nu  + 1}}} \right\|_{\lambda {\sigma _\nu }}} &= {\left\| {D\Psi _{\nu  + 1}^{ - 1} \cdot {\Delta ^ * }{P_{\nu  + 1}} \circ {\Psi _{\nu  + 1}}} \right\|_{\lambda {\sigma _\nu }}} \leqslant 2{\left\| {{\Delta ^ * }{P_{\nu  + 1}}} \right\|_{\alpha {\sigma _\nu }}}\notag \\
& = 2{\Bigg\| {\sum\limits_{k \in {\mathscr{K}_{\nu  + 1}}} {\hat P(k){{\rm e}^{{\rm i}k \cdot x}}} } \Bigg\|_{\alpha {\sigma _\nu }}} \leqslant 2\sum\limits_{k \in {\mathscr{K}_{\nu  + 1}}} {|\hat P(k)|{{\rm e}^{\alpha {\sigma _\nu }{K_{\nu  + 1}}}}}\notag \\
& \leqslant 2{{\rm e}^{\alpha b\left( {r + {d_\nu }} \right)}}\sum\limits_{k \in {\mathscr{K}_{\nu  + 1}}} {|\hat P(k)|}  \leqslant \frac{{2{{\rm e}^{\alpha b\left( {r + {d_\nu }} \right)}}}}{{m\left( {{b^\nu }} \right)}}\sum\limits_{k \in {\mathscr{K}_{\nu  + 1}}} {|\hat P(k)|{m_k}}\notag \\
\label{VECPSIDELTAP}& = \frac{{2{{\rm e}^{{2^{ - 1}}\left( {1 + \lambda } \right)b\left( {r + {d_\nu }} \right)}}{{\left\| {{\Delta ^ * }{P_{\nu  + 1}}} \right\|}_m}}}{{m\left( {{b^\nu }} \right)}},
\end{align} 
where $ \alpha  = {2^{ - 1}}\left( {{\lambda} + 1} \right) $, and the fact $ {K_{\nu  + 1}} = b{K_\nu } $ is used. \textit{This turns the usual exponential weighted norm depending on the iterative sequence in KAM theory to the $ m $-weighted norm introduced in this paper, which allows us to determine the concrete regularity required for the  perturbation more feasible and convenient.} Now, it follows from \eqref{VECQNU+} and \eqref{VECPSIDELTAP} that
\begin{align*}
{\left\| {{Q_{\nu  + 1}}} \right\|_{{\sigma _{\nu  + 1}}}} & \leqslant {\left\| {Q_\nu ^ + } \right\|_{\lambda {\sigma _\nu }}} + {\left\| {\Psi _{\nu  + 1}^ * {\Delta ^ * }{P_{\nu  + 1}}} \right\|_{\lambda {\sigma _\nu }}}\\
&  \leqslant q{{\rm e}^{ - \gamma {d_\nu }}}{\left\| {{Q_\nu }} \right\|_{{\sigma _\nu }}} + \frac{{2{{\rm e}^{{2^{ - 1}}\left( {1 + \lambda } \right)b\left( {r + {d_\nu }} \right)}}{{\left\| {{\Delta ^ * }{P_{\nu  + 1}}} \right\|}_m}}}{{m\left( {{b^\nu }} \right)}}\\
& \leqslant q{{\rm e}^{ - \gamma {d_\nu }}}{\varepsilon _\nu } + \frac{{2{{\rm e}^{{2^{ - 1}}\left( {1 + \lambda } \right)b{d_\nu }}}{{\left\| {{\Delta ^ * }{P_{\nu  + 1}}} \right\|}_m}}}{{m\left( {{b^\nu }} \right)}}\\
&  = C\sum\limits_{\mu  = 0}^\nu  {{q^{\nu  + 1 - \mu }}\exp \left( { - \gamma \sum\limits_{j = \mu }^\nu  {{d_j}} } \right)\frac{{{{\rm e}^{{2^{ - 1}}\left( {1 + \lambda } \right)b{d_{\mu  - 1}}}}{{\left\| {{\Delta ^ * }{P_\mu }} \right\|}_m}}}{{m\left( {{b^{\mu  - 1}}} \right)}}}  + \frac{{2{{\rm e}^{{2^{ - 1}}\left( {1 + \lambda } \right)b{d_\nu }}}{{\left\| {{\Delta ^ * }{P_{\nu  + 1}}} \right\|}_m}}}{{m\left( {{b^\nu }} \right)}}\\
&  \leqslant C\sum\limits_{\mu  = 0}^{\nu  + 1} {{q^{\nu  + 1 - \mu }}\exp \left( { - \gamma \sum\limits_{j = \mu }^\nu  {{d_j}} } \right)\frac{{{{\rm e}^{{2^{ - 1}}\left( {1 + \lambda } \right)b{d_{\mu  - 1}}}}{{\left\| {{\Delta ^ * }{P_\mu }} \right\|}_m}}}{{m\left( {{b^{\mu  - 1}}} \right)}}} \\
& = {\varepsilon _{\nu  + 1}},
\end{align*}
which completes the induction for $ {\left\| {{Q_\nu }} \right\|_{{\sigma _\nu }}} \leqslant {\varepsilon _\nu } $ in \eqref{VECdiedai1}.

 Moreover, note that both $ \Phi_\nu $ and $ \Psi_\nu $ are nearly identical transformations, and $ {\Psi _{\nu  + 1}}  = {\Psi _\nu } \circ {\Phi _\nu } $. Therefore,  direct calculation gives
\begin{align}
D{\Psi _{\nu  + 1}} - \mathbb{I}	& = D{\Psi _\nu } \circ {\Phi _\nu } \cdot D{\Phi _\nu } - \mathbb{I}\notag \\
	& = D{\Psi _\nu } \circ {\Phi _\nu } + D{\Psi _\nu } \circ {\Phi _\nu } \cdot (D{{ \Phi }_\nu }-\mathbb{I}) - \mathbb{I} \circ {\Phi _\nu }\notag \\
	\label{HUAJIANDEL} & = (D{{ \Psi }_\nu }-\mathbb{I}) \circ {\Phi _\nu } + D{\Psi _\nu } \circ {\Phi _\nu } \cdot (D{{ \Phi }_\nu }-\mathbb{I}).
\end{align} 
Similar to that in \eqref{VECHATFAI}, by  \eqref{VECdiedai2}, \eqref{VECSTEP2}  and the fact that $ r $ is sufficiently large, we get
\begin{equation}\label{recall}
	\|{{ \Phi }_\nu }-{\rm id}\|_{{\alpha{\sigma_\nu }}} \leqslant \frac{{4{\Delta _\nu }{\varepsilon _\nu }}}{{{K_\nu }}} \leqslant \frac{{4\alpha\left( {r + {d_\nu }} \right)}}{{{K_\nu }}} = \alpha{\sigma_\nu }\leqslant \sigma_\nu.
\end{equation}
Hence, it follows from the Neumann Lemma \ref{VECL1} and \eqref{VECSTEP2} that
\begin{align*}
	{\| D{\Psi _{\nu  + 1}} - \mathbb{I} \|_{{\sigma_{\nu  + 1}}}} &\leqslant {\| {D{\Psi _{\nu  }} - \mathbb{I}} \|_{{\sigma_\nu }}} + {\left\| {D{\Psi _\nu }} \right\|_{{\sigma_\nu }}}{\| {D{{ \Phi }_\nu }-\mathbb{I}} \|_{{\sigma_{\nu  + 1}}}}\\
	& \leqslant {\delta _\nu } + \left( {1 + {\delta _\nu }} \right) \cdot 4{\Delta _\nu }{\varepsilon _\nu }\\
	& = \left( {1 + {\delta _\nu }} \right)\left( {1 + 4{\Delta _\nu }{\varepsilon _\nu }} \right) - 1\\
	& = \left( {\prod\limits_{\mu  = 0}^{\nu  - 1} {\left( {1 + 4{\Delta _\mu }{\varepsilon _\mu }} \right)} } \right)\left( {1 + 4{\Delta _\nu }{\varepsilon _\nu }} \right) - 1\\
	& = \prod\limits_{\mu  = 0}^\nu  {\left( {1 + 4{\Delta _\mu }{\varepsilon _\mu }} \right)}  - 1\\
	& = {\delta _{\nu  + 1}},
\end{align*}
which completes the proof of \eqref{VECdiedai2} for all $ \nu \in \mathbb{N} $.

As for the proof of  \eqref{VECdiedai3}, it suffices to recall the relation $ {\tilde\omega_{\nu  + 1}} = {\tilde\omega_\nu } + {\omega'_\nu } $ and the estimate \eqref{VECSTEP1} derived from the Step Lemma \ref{VECSTEP}.

Finally, direct calculation yields
\begin{align}
	D{\Psi _{\nu  + 1}} - D{\Psi _\nu }  &= D\left( {{\Psi _\nu } \circ {\Phi _\nu }} \right) - D{\Psi _\nu }\notag \\
	& = D{\Psi _\nu } \circ {\Phi _\nu } \cdot D{\Phi _\nu } - D{\Psi _\nu }\notag \\
	\label{VECZUIHOU} & = \left( {D{\Psi _\nu } \circ {\Phi _\nu } - D{\Psi _\nu }} \right) + ( {D{\Psi _\nu } \circ {\Phi _\nu } \cdot (D{{ \Phi }_\nu }-\mathbb{I})} ).
\end{align}
For the first term in \eqref{VECZUIHOU}, the Cauchy type Lemma \ref{VECCAULE},  along with  \eqref{VECSTEP2}, provides the following estimate:
\begin{align}
	{\left\| {D{\Psi _\nu } \circ {\Phi _\nu } - D{\Psi _\nu }} \right\|_{0}} &\leqslant {\left\| {{D^2}{\Psi _\nu }} \right\|_{0}}{\left\| { {\Phi _\nu }-{\rm id}} \right\|_0} \leqslant \frac{{1}}{{\rm e}{{\sigma_\nu }}}{\left\| {D{\Psi _\nu }} \right\|_{{\sigma_{\nu  }}}}{\left\| {{\Phi _\nu }-{\rm id}} \right\|_{0}}\notag \\
	\label{VECzuihou1}& \leqslant \frac{{2}}{{\rm e}{{\sigma_\nu }}}{\| {{{ \Phi_\nu -{{\rm id}}} }} \|_{\alpha{\sigma_\nu }}} \leqslant \frac{{2{\Delta _\nu }{\varepsilon _\nu }}}{{\rm e}{{\sigma_\nu }{K_\nu }}} \leqslant \frac{{2}}{{\rm e}r}{\Delta _\nu }{\varepsilon _\nu } \leqslant 4{\Delta _\nu }{\varepsilon _\nu },
\end{align}
provided that $ r \geqslant (2{\rm e})^{-1} $. For the second term in \eqref{VECZUIHOU}, it follows from \eqref{VECdiedai2} and \eqref{VECSTEP2} that
\begin{equation}\label{VECzuihou2}
	{\| {D{\Psi _\nu } \circ {\Phi _\nu } \cdot (D{{ \Phi }_\nu }-\mathbb{I})} \|_{0}} \leqslant {\left\| {D{\Psi _\nu }} \right\|_{{\alpha \sigma_\nu }}}{\| {D{{ \Phi }_\nu }}-\mathbb{I} \|_{0}} \leqslant {\Delta _\nu } \cdot 4{\Delta _\nu }{\varepsilon _\nu } \leqslant 8{\Delta _\nu }{\varepsilon _\nu },
\end{equation}
because $ 0<\varepsilon \ll 1 $. Now, substituting \eqref{VECzuihou1} and \eqref{VECzuihou2} into \eqref{VECZUIHOU}, one finally proves \eqref{VECdiedai4}:
\[{\left\| {D{\Psi _{\nu  + 1}} - D{\Psi _\nu }} \right\|_{0}} \leqslant 12{\Delta _\nu }{\varepsilon _\nu }.\]
\end{proof}

\subsubsection{The uniform convergence}\label{VEC Uniform convergence}
The uniform convergence in the Iterative Lemma \ref{VECIterative} for KAM conjugacy is easy to see, that is, the modifying terms $ \tilde \omega _\nu $ admit a unique limit denoted as $ \tilde \omega $, the  transformations $ \Psi_\nu $ have a unique $ C^1 $  limit $ \Psi $ ($ \Psi $ is at least $ C^1 $ because $ D\Psi $ is continuous, and it  might admit higher regularity), and $ Q_\nu \to 0$ in the $ \| \cdot \|_0 $ norm. As a consequence, the limit form of \eqref{VECshoulian} is 
\[{\Psi ^ * }\left( {\omega  - \tilde \omega  + P} \right) = \omega ,\]
which proves the desired KAM conjugacy in the Abstract $ m $-weighted KAM Theorem \ref{VECABSTRACT} for the infinite-dimensional  vector field $ \omega $.

\subsection{Proof of Theorem \ref{VECT1}}\label{VECGE}
We shall employ the Abstract $ m $-weighted KAM Theorem \ref{VECABSTRACT} to prove Theorem \ref{VECT1}. For given $ b,\eta >1 $,  set the balancing sequence as $ {d_j} = {c_1}j{\theta ^j}$ with $ \theta  = {b^{\frac{1}{{1 + \eta }}}} \in \left( {1,b} \right) $, where $ c_1 \gg 1 $ is an undetermined number. Then, it is evident to  verify that $ \lambda  = {b^{ - 1}}\mathop {\overline {\lim } }\nolimits_{\nu  \to \infty } d_\nu ^{ - 1}{d_{\nu  + 1}} = {b^{ - 1}}\theta  $ and $ \gamma  = \frac{3}{4}\left( {1 - \lambda } \right) = \frac{3}{4}\left( {1 - {b^{ - 1}}\theta } \right) $. On the one hand, it follows that
\begin{equation}\label{VECT1VERIFY}
\exp \left( { - \gamma \sum\limits_{j = 0}^{\nu  - 1} {{d_j}} } \right) = \exp \left( { - {c_1}\gamma \sum\limits_{j = 0}^{\nu  - 1} {j{\theta ^j}} } \right) = \mathcal{O}\left( {\exp \left( { - {c_1}{c_2}\nu {\theta ^\nu }} \right)} \right),
\end{equation}
provided with some $ c_2>0 $ independent of $ c_1>0$.
On the other hand, the Diophantine small divisor Lemma \ref{VECSM} yields the small divisor estimate as
\begin{equation}\label{VECT1Y}
	{b^\nu }\mathop {\max }\limits_{0 < {{\left| k \right|}_\eta } \leqslant {b^\nu }} {\left| {k \cdot \omega } \right|^{ - 1}} \leqslant {b^\nu }\exp \left( {C\left( {\eta ,\mu ,b} \right)\nu {b^{\frac{\nu }{{1 + \eta }}}}} \right).
\end{equation}
Hence, if $ {c_1} \geqslant c_2^{ - 1}C\left( {\eta ,\mu ,b} \right) $, then \eqref{VECT1VERIFY} and  \eqref{VECT1Y}  ensure  the boundedness required by Theorem \ref{VECABSTRACT} for all $ 0 < q < {b^{ - 1}} <1$:
\begin{align}
\sum\limits_{\nu  = 0}^\infty  {{q^\nu }\exp \left( { - \gamma \sum\limits_{j = 0}^{\nu  - 1} {{d_j}} } \right){b^\nu }\mathop {\max }\limits_{0 < {{\left| k \right|}_\eta }{b^\nu }} {{\left| {k \cdot \omega } \right|}^{ - 1}}} 
& = \mathcal{O}\left( {\sum\limits_{\nu  = 0}^\infty  {{q^\nu }\mathcal{O}\left( {\exp \left( { - {c_1}{c_2}\nu {\theta ^\nu }} \right)} \right) \cdot {b^\nu }\exp \left( {C\left( {\eta ,\mu ,b} \right)\nu {\theta ^\nu }} \right)} } \right)\notag \\
 &= \mathcal{O}\left( {\sum\limits_{\nu  = 0}^\infty  {{{\left( {qb} \right)}^\nu } \cdot \exp \left( { - \left( {{c_1}{c_2} - C\left( {\eta ,\mu ,b} \right)} \right)\nu {\theta ^\nu }} \right)} } \right)\notag \\
&= \mathcal{O}\left( {\sum\limits_{\nu  = 0}^\infty  {{{\left( {qb} \right)}^\nu }} } \right)\notag \\
\label{VECT1P} &= \mathcal{O}\left( 1 \right),
\end{align}
which implies the existence of the weight $ m $ admitted by the perturbation $ P $.
More explicitly, by recalling (II) in Theorem \ref{VECABSTRACT}, the weight  could be chosen as 
\[{m_k} = m\left( {{{\left| k \right|}_\eta }} \right) = \exp \left( {{C_{\rm G1}}\ln \left( {1 + {{\left| k \right|}_\eta }} \right)\left| k \right|_\eta ^{\frac{1}{{1 + \eta }}}} \right)\]
with some $ C_{\rm G1}\gg 1 $. To see this, we assert that
\[\mathop {\overline {\lim } }\limits_{\mu  \to \infty } \frac{{\rho \left( \mu  \right){{\rm e}^{{2^{ - 1}}\left( {1 + \lambda } \right)b{d_{\mu  - 1}}}}}}{{m\left( {{b^{\mu  - 1}}} \right)}} <  + \infty .\]
It is evident to observe that
\[{{\mathrm{e}}^{{2^{ - 1}}\left( {1 + \lambda } \right)b{d_{\mu  - 1}}}} = \exp \left( {\mathcal{O}\left( {\left( {\mu  - 1} \right){\theta ^{\mu  - 1}}} \right)} \right),\;\;\mu  \gg 1,\]
and with the help of \eqref{VECT1P}, one has
\begin{align}
	\rho \left( \mu  \right) &= \sum\limits_{\nu  = \mu }^\infty  {{q^{\nu  - \mu }}\exp \left( { - \gamma \sum\limits_{j = \mu }^{\nu  - 1} {{d_j}} } \right){b^\nu }\mathop {\max }\limits_{0 < {{\left| k \right|}_\eta \leqslant }{b^\nu }} {{\left| {k \cdot \omega } \right|}^{ - 1}}} \notag \\
& = \sum\limits_{\nu  = \mu }^\infty  {{q^{ - \mu }}\exp \left( {\gamma \sum\limits_{j = 0}^{\mu  - 1} {{d_j}} } \right) \cdot {q^\nu }\exp \left( { - \gamma \sum\limits_{j = 0}^{\nu  - 1} {{d_j}} } \right){b^\nu }\mathop {\max }\limits_{0 < {{\left| k \right|}_\eta } \leqslant {b^\nu }} {{\left| {k \cdot \omega } \right|}^{ - 1}}} \notag \\
\label{VECT1PR}&  = \mathcal{O}\left( {{q^{ - \mu }}\exp \left( {\gamma \sum\limits_{j = 0}^{\mu  - 1} {{d_j}} } \right)} \right) = \exp \left( {\mathcal{O}\left( {\left( {\mu  - 1} \right){\theta ^{\mu  - 1}}} \right)} \right),\;\;\mu  \gg 1.
\end{align}
The above two terms could be well dominated by  $ m\left( {{b^{\mu  - 1}}} \right) $, since
\[m\left( {{b^{\mu  - 1}}} \right) = \exp \left( {{C_{\rm{G1}}}\ln \left( {1 + {b^{\mu  - 1}}} \right){b^{\frac{{\mu  - 1}}{{1 + \eta }}}}} \right) \geqslant \exp \left( {\left( {{C_{{\rm{G1}}}}\ln b} \right)\left( {\mu  - 1} \right){\theta ^{\mu  - 1}}} \right),\]
as promised. Finally, the regularity  for the perturbation $ P  $ can be directly derived from $ m_k $ by utilizing the Abstract $ m $-weighted KAM Theorem \ref{VECABSTRACT}. Specifically, assuming that
\[{\left\| P \right\|_m} = \sum\limits_{0 \ne k \in \mathbb{Z}_*^\infty } {|\hat P(k)|\exp \left( {{C_{{\rm{G1}}}}\ln \left( {1 + {{\left| k \right|}_\eta }} \right)\left| k \right|_\eta ^{\frac{1}{{1 + \eta }}}} \right)}  \ll 1\]
for some $ {C_{{\rm{G1}}}} \gg 1 $ is sufficient to ensure  the KAM conjugacy for the vector field $ \omega $ with the infinite-dimensional Diophantine nonresonance as defined in Definition \ref{VECINDIO}. This establishes the desired result.

\begin{remark}
One may observe that our choices for the balancing sequence $ \{d_\nu\}_{\nu \in \mathbb{N}} $ and the  weight $ m $ are somewhat rough. It can indeed be quantitatively  improved by employing more elaborate asymptotic analysis. However, we prefer not to delve into this here for the sake of simplicity.
\end{remark}
\subsection{Proofs of Theorems \ref{VECT2} and  \ref{VECT33}}\label{VECPT2}
The proofs are based on the observation that the weighted norm for the vector $ k \in \mathbb{Z}_*^\infty $ employed in  our Abstract $ m $-weighted KAM Theorem \ref{VECABSTRACT}, namely $ {{\left| k  \right|}_\eta }: = \sum\nolimits_{j \in \mathbb{Z}} {{{\left\langle j \right\rangle }^\eta }\left| {{k _j}} \right|}  <  + \infty $, could be replaced by any  norm $ {\left\| k \right\|_w}: = \sum\nolimits_{j \in \mathbb{Z}} {w\left( {\left\langle j \right\rangle } \right)\left| {{k_j}} \right|} $ with a certain approximation function $ w $, because we only use the property $ \left| k \right| \leqslant {\left| k \right|_\eta } $ throughout the analysis, see lemmas given in Section \ref{VECLEMMAS} for instance. Now, it is crucial to construct a nonresonant condition for which almost all frequencies in $ \mathbb{R}^\mathbb{Z} $ hold. For   weight (approximation functions) $ \Delta $ and $ w $ to be determined, set the small divisor condition as
\[\left| {k \cdot \omega } \right| > \frac{{{\gamma ^ * }}}{{\Delta \left( {{{\left\| k \right\|}_w}} \right)}},\;\;{\gamma ^ * } > 0,\;\;0 \ne k \in \mathbb{Z}_ * ^\infty .\]
It suffices to require $ \sum\nolimits_{0 \ne k \in \mathbb{Z}_ * ^\infty } {\frac{{{\gamma ^ * }}}{{\Delta \left( {{{\left\| k \right\|}_w}} \right)}}}  = o\left( 1 \right) $ as $ {\gamma ^ * } \to {0^ + } $. Now we are in a position to establish a cardinality estimate. For $ \nu \gg 1 $, we observe that the largest $ j $ such that $ |k_j| \ne 0 $ under assumption $ {\left\| k \right\|_w} \in [\nu  - 1,\nu ) $ must satisfy $ {j_{\max }} \leqslant {w^{ - 1}}\left( \nu  \right) $, because $ w\left( {\left\langle j \right\rangle } \right) \leqslant w\left( {\left\langle j \right\rangle } \right)\left| {{k_j}} \right| \leqslant {\left\| k \right\|_w} \leqslant \nu  $. As a consequence, it holds
\begin{align}
&\# \left\{ {0 \ne k \in \mathbb{Z}_ * ^\infty :{{\left\| k \right\|}_w} \in [\nu  - 1,\nu )} \right\}\notag \\
 \leqslant &\# \left\{ {0 \ne k \in \mathbb{Z}_ * ^\infty :\left| {{k_0}} \right| + \left| {{k_{ - 1}}} \right| + \left| {{k_1}} \right| +  \cdots  + |{k_{ - [{w^{ - 1}}\left( \nu  \right)]}}| + |{k_{[{w^{ - 1}}\left( \nu  \right)]}}| = \nu } \right\}\notag \\
 \leqslant &{2^{2[{w^{ - 1}}\left( \nu  \right)] + 1}}\# \left\{ {{k_j} \in \mathbb{N}:{k_0} + {k_{ - 1}} + {k_1} +  \cdots  + {k_{ - [{w^{ - 1}}\left( \nu  \right)]}} + {k_{[{w^{ - 1}}\left( \nu  \right)]}} = \nu } \right\}\notag \\
\label{VECJISHU} \leqslant &{2^{2[{w^{ - 1}}\left( \nu  \right)] + 1}}C_{\nu  + [{w^{ - 1}}\left( \nu  \right)]}^\nu .
\end{align}
With \eqref{VECJISHU} and $ \gamma^* \ll 1 $, one gets
\begin{align}
\sum\limits_{0 \ne k \in \mathbb{Z}_ * ^\infty } {\frac{{{\gamma ^ * }}}{{\Delta \left( {{{\left\| k \right\|}_w}} \right)}}}  &= {\gamma ^ * }\sum\limits_{\nu  \geqslant 0} {\sum\limits_{{{\left\| k \right\|}_w} \in [\nu  - 1,\nu )} {\frac{1}{{\Delta \left( {{{\left\| k \right\|}_w}} \right)}}} }  \notag \\
\label{VECYJ}&= {\gamma ^ * }\mathcal{O}\left( {\sum\limits_{\nu  \geqslant 0} {\frac{{{2^{2[{w^{ - 1}}\left( \nu  \right)] + 1}}C_{\nu  + [{w^{ - 1}}\left( \nu  \right)]}^\nu }}{{\Delta \left( \nu  \right)}}} } \right) =\mathcal{O}\left( {{\gamma ^ * }} \right)=o(1)
\end{align}
as desired, whenever  $ \Delta $ is larger than an arbitrarily given polynomial's type, that is, $ \Delta \left( x \right) \gg {x^L} $ with any $ L>0 $, and $ w $ depending on $ \Delta $ could be suitably chosen to increase rapid enough, thus in this case $ [{w^{ - 1}}\left( \nu  \right)] $ is  small (but tends to $ +\infty $) and it could be killed by $ \Delta {\left( \nu  \right)^{ - 1}} $, we therefore obtain the   boundedness of the series in \eqref{VECYJ} (as well as the global smallness if $ \gamma^* \ll 1 $). As a remark, in this infinite-dimensional case, $ \Delta \left( x \right) \gg {x^L} $ is essential due to the counterexamples (as the dimension tends to $ +\infty $) constructed by Herman \cite{MR0874026} and  etc. This can also be seen in our argument, that is, assuming $ \Delta \left( x \right) = {x^L} $ with some $ L>0 $, we have to require that $ [{w^{ - 1}}\left( \nu  \right)] = \mathcal{O}\left( 1 \right) $ by the properties of the combinations, which contradicts with the assumption on the weight (approximation function) $ w$, since it must tend to $ +\infty $. It is also worth emphasizing that, the Jackson type approximation theorem for finitely differentiable maps admits an \textit{optimal} error estimate, and the control constant depending on the dimension $ n $ indeed tends to $ +\infty $ as $ n \to +\infty $. Therefore, even if there does exist a  nonresonant condition of the polynomial's type under the infinite-dimensional setting, one has no way to deal with them.

Based on the above preparation, we are now ready to choose an \textit{explicit}  approximation function $ \Delta $ admitted by almost all $ \omega \in \mathbb{R}^{\mathbb{Z}} $. Although the implicit case could be weaker, we prefer to choose an explicit one in order to highlight the result and avoid complexity. As an illustration, let $ \Delta \left( x \right) \sim \exp \left( {{{\left( {\ln x} \right)}^a}} \right) $ with some $ a>1 $. One easily verifies that it is larger than any polynomial's type. It remains to construct the regularity for the perturbation via (I) and (II) in Theorem \ref{VECABSTRACT}. For $b>1$, setting $ {d_j} = A{j^{a - 1}} $ yields $ \lambda  = {b^{ - 1}}\mathop {\overline {\lim } }\nolimits_{j \to \infty } d_j^{ - 1}{d_{j + 1}} = {b^{ - 1}} $ and $ \gamma  = \frac{3}{4}\left( {1 - \lambda } \right) = \frac{3}{4}\left( {1 - {b^{ - 1}}} \right) $. Then we let $A > a{\gamma ^{ - 1}}{\left( {\ln b} \right)^a}  $. Besides, one easily checks that
\begin{equation}\label{VECT2PIE}
	A\gamma {a^{ - 1}}{\left( {\nu  - 1} \right)^a} = A\gamma \int_0^{\nu  - 1} {{s^{a - 1}}{\rm d}s}  \leqslant \gamma \sum\limits_{j = 0}^{\nu  - 1} {{d_j}}  \leqslant A\gamma \int_0^\nu  {{s^{a - 1}}{\rm d}s}  = A\gamma {a^{ - 1}}{\nu ^a}.
\end{equation}
In this case, assuming $ q \in \left( {0,{b^{ - 1}}} \right) $ leads to the boundedness in (I)
\begin{align*}
&\sum\limits_{\nu  = 0}^\infty  {{q^\nu }\exp \left( { - \gamma \sum\limits_{j = 0}^{\nu  - 1} {{d_j}} } \right){b^\nu }\mathop {\max }\limits_{0 < {{\left\| k \right\|}_w} \leqslant {b^\nu }} {{\left| {k \cdot \omega } \right|}^{ - 1}}} \\
\leqslant & {\gamma ^ * }^{ - 1}\sum\limits_{\nu  = 0}^\infty  {{{\left( {qb} \right)}^\nu }\exp \left( { - A\gamma {a^{ - 1}}{{\left( {\nu  - 1} \right)}^a} + {{\left( {\ln b} \right)}^a}{\nu ^a}} \right)} \\
= &\mathcal{O}\left( {\sum\limits_{\nu  = 0}^\infty  {{{\left( {qb} \right)}^\nu }} } \right) = \mathcal{O}\left( 1 \right)
\end{align*}
due to our choice of $ A $. Moreover, similar to the argument in \eqref{VECT1PR}, we derive  
\[\rho \left( \mu  \right) = \mathcal{O}\left( {\exp \left( { - \gamma \sum\limits_{j = \mu }^{\nu  - 1} {{d_j}} } \right){b^\nu }} \right) = \mathcal{O}\left( {\exp \left( {\left( {{{\left( {\ln b} \right)}^a} + \epsilon } \right){\mu ^a}} \right)} \right)\]
for any  $ \epsilon>0 $, and $ {{\rm e}^{{2^{ - 1}}\left( {1 + \lambda } \right)b{d_{\mu  - 1}}}} = o\left( {\exp \left( {{\mu ^a}} \right)} \right) $. Note that $ b>1 $ and $ \epsilon>0 $ could be arbitrarily chosen, therefore for any $ C_{\rm L1}>1 $, letting $ {m_k} = m\left( {{{\left\| k \right\|}_w}} \right) \sim \exp \left( {C_{\rm L1} {{\left( {\ln {{\left\| k \right\|}_w}} \right)}^a}} \right) $ is sufficient to obtain the relation
\[\mathop {\overline {\lim } }\limits_{\mu  \to \infty } \frac{{\rho \left( \mu  \right){{\rm e}^{{2^{ - 1}}\left( {1 + \lambda } \right)b{d_{\mu  - 1}}}}}}{{m\left( {{b^{\mu  - 1}}} \right)}} <  + \infty \]
in (II), and this yields the explicit regularity requirement  for the perturbation:
\[\sum\limits_{0 \ne k \in \mathbb{Z}_ * ^\infty } {|\hat P(k)|\exp \left( {C_{\rm L1} {{\left( {\ln {{\left\| k \right\|}_w}} \right)}^a}} \right)}  \ll 1.\]
Finally, one easily checks that $ C_{\rm L1}>1 $ could be replaced by $ C_{\rm L1}'>0 $ in the above conclusion, if we modify the approximation function $ \Delta \left( x \right) \sim \exp \left( {{{\left( {\ln x} \right)}^a}} \right) $ by $ \Delta \left( x \right) \sim \exp \left( {{{\left( C_{\rm L1}''{\ln x} \right)}^a}} \right) $ for sufficiently small $ C_{\rm L1}''>0 $. This proves Theorem \ref{VECT2}.

As for Theorem \ref{VECT33}, the analysis is indeed similar, as long as we note that the weight $ u $ depends on the given `super-polynomial' $ \Delta $, as previously shown.

\subsection{Proof of Theorem \ref{VECT3}}\label{VECYOUXIAN}
Here, we would like to emphasize again that the proof of  the Abstract $ m $-weighted KAM Theorem \ref{VECABSTRACT} is not closely related to the specific form of the nonresonance itself (such as the infinite-dimensional Diophantine nonresonance in Definition \ref{VECINDIO}), but rather strongly depends on the small divisor estimate it leads to (such as the Diophantine small divisor Lemma \ref{VECSM}). Based on this observation, we shall now proceed to present the proof  for the quasi-periodic case.

We first consider the proof of (i). It is well known that frequencies satisfying the finite-dimensional Diophantine nonresonance in \eqref{VECFINITEDIO} do exist, and admit full Lebesgue measure in $ \mathbb{R}^n $. Similar to the  arguments in Section \ref{VECPT2},  the Abstract $ m $-weighted KAM Theorem \ref{VECABSTRACT} does work for the finite-dimensional case, because $ \left| k \right| \leqslant \left| k \right| $ (note that the weight $ w $ being some  approximation function is only required for the infinite-dimensional case). Therefore, with the Diophantine assumption \eqref{VECFINITEDIO} for the  vector field $ \omega \in \mathbb{R}^n$, setting the balancing sequence $ d_j\equiv0 $ yields 
\[\sum\limits_{\nu  = 0}^\infty  {{q^\nu }\exp \left( { - \gamma \sum\limits_{j = 0}^{\nu  - 1} {{d_j}} } \right){b^\nu }\mathop {\max }\limits_{0 < \left| k \right| \leqslant {b^\nu }} {{\left| {k \cdot \omega } \right|}^{ - 1}}}  \leqslant {\gamma ^ * }^{ - 1}\sum\limits_{\nu  = 0}^\infty  {{q^\nu }{b^{\left( {\beta  + 1} \right)\nu }}}  = \mathcal{O}\left( 1 \right)\]
for all $ q \in \left( {0,{b^{ - \beta  - 1}}} \right) $ and $ b>1 $. Besides, by the exponential property, one obtains that
\begin{align*}
\rho \left( \mu  \right)& = \sum\limits_{\nu  = \mu }^\infty  {{q^{\nu  - \mu }}\exp \left( { - \gamma \sum\limits_{j = \mu }^{\nu  - 1} {{d_j}} } \right){b^\nu }\mathop {\max }\limits_{0 < \left| k \right| \leqslant {b^\nu }} {{\left| {k \cdot \omega } \right|}^{ - 1}}} \\
&= \sum\limits_{\nu  = \mu }^\infty  {{q^{\nu  - \mu }}{b^{\left( {\beta  + 1} \right)\nu }}}  = {q^{ - \mu }}\sum\limits_{\nu  = \mu }^\infty  {{{\left( {q{b^{\beta  + 1}}} \right)}^\nu }}  = \mathcal{O}\left( {{q^{ - \mu }} \cdot {{\left( {q{b^{\beta  + 1}}} \right)}^\mu }} \right) = \mathcal{O}\left( {{b^{\left( {\beta  + 1} \right)\mu }}} \right),
\end{align*}
which shows that taking the weight $ m $ as $ {m_k} = m\left( {\left| k \right|} \right) = {\left| k \right|^{\beta  + 1}} $ is sufficient to achieve the requirement of (II) in Theorem \ref{VECABSTRACT}. Therefore, the KAM conjugacy could be proved whenever $ \sum\nolimits_{0 \ne k \in {\mathbb{Z}^n}} {|\hat P(k)|{{\left| k \right|}^{\beta+1}}} \ll 1 $, which establishes (i).

As for the proofs of (ii) and (iii), it is  evident that the almost periodic results, namely Theorem \ref{VECT1} and Theorem \ref{VECT2}, could be valid for the quasi-periodic case, provided  one appropriately  modifies the nonresonant conditions and  replaces $ |k|_\eta $ with $ |k| $.   A primary distinction between the finite-dimensional and infinite-dimensional Gevrey cases lies in the form of the small divisor estimates. For the infinite-dimensional Diophantine nonresonance, as shown in the Diophantine small divisor Lemma \ref{VECSM}, the estimate is of the form  (note that $0< \frac{1}{1+\eta} <1$):
 	\begin{equation}\label{VECJJJ}
 			\mathop {\max }\limits_{0 < \left| k \right|_\eta \leqslant N} {\left| {k \cdot \omega } \right|^{ - 1}} \leqslant {\gamma ^ * }^{-1} \mathop {\sup }\limits_{0<{{\left| k \right|}_\eta } < N} \prod\limits_{j \in \mathbb{Z}} {\left( {1 + {{\left\langle j \right\rangle }^\mu }{{\left| {{k_j}} \right|}^\mu }} \right)}  \leqslant {\gamma ^ * }^{-1} \exp \left( {C\left( {\eta ,\mu } \right){N^{\frac{1}{{1 + \eta }}}}\ln \left( {1 + N} \right)} \right),
 	\end{equation}
  while that based on the finite-dimensional  nonresonance via an approximation function $ \Delta \left( x \right) \sim \exp \left( {{x^\zeta }} \right)  $ with $ 0<\zeta <1$ is of the form:
  \begin{equation}\label{VECYXG}
  	\mathop {\max }\limits_{0 < \left| k \right| \leqslant N} {\left| {k \cdot \omega } \right|^{ - 1}} \leqslant {\gamma ^ * }^{-1}\mathop {\max }\limits_{0 < \left| k \right| \leqslant N} \Delta \left( {\left| k \right|} \right) \leqslant {\gamma ^ * }^{-1} \Delta \left( N \right) \sim {\gamma ^ * }^{-1}\exp \big( {{N^\zeta }} \big).
  \end{equation}
The two forms differ by an order of $\ln(1+N)$ within the exponential function. To address this, a balancing sequence
 \[{d_j} = {c_1}j{\theta ^j},\;\;  \theta  = {b^{\frac{1}{{1 + \eta }}}} \in \left( {1,b} \right) ,\;\; c_1 \gg 1, \;\;j \in \mathbb{N}\]
 was introduced in the analysis of the former (Section \ref{VECGE}). The purpose of adding the term  $ j $ before $ \theta^j $ is to match the aforementioned order of $ \ln(1+N) $. However, in the finite-dimensional case, this order of $ \ln(1+N) $ is not present, as shown in \eqref{VECYXG}, hence the balancing sequence can be  chosen simply as 
   \[{d_j} = {c_1}{\theta ^j},\;\;  \theta  = {b^{\frac{1}{{1 + \eta }}}} \in \left( {1,b} \right) ,\;\; c_1 \gg 1, \;\;j \in \mathbb{N}.\]
Under these considerations, it is evident  from the proof of Theorem \ref{VECGE} that choosing 
   \[{m_k} = m\left( {{{\left| k \right|} }} \right) = \exp \left( {{C_{\rm G2}}\left| k \right| ^{\frac{1}{{1 + \eta }}}} \right),\;\; C_{\rm G2} \gg 1\]
    will ensure the KAM conjugacy. To be more precise, we have 
      \[{\rho \left( \mu  \right){{\rm e}^{{2^{ - 1}}\left( {1 + \lambda } \right)b{d_{\mu  - 1}}}}}= \exp \left( \mathcal{O}\left({  {\theta ^{\mu  - 1}}}\right) \right),\]
   and
  \[m\left( {{b^{\mu  - 1}}} \right) = \exp \left( {{C_{\rm{G2}}}{b^{\frac{{\mu  - 1}}{{1 + \eta }}}}} \right)= \exp \left( { {{C_{{\rm{G2}}}}} {\theta ^{\mu  - 1}}} \right) \gg {\rho \left( \mu  \right){{\rm e}^{{2^{ - 1}}\left( {1 + \lambda } \right)b{d_{\mu  - 1}}}}}.\]
    This leads to the regularity of the perturbation $ P $ being:
   \[{\left\| P \right\|_m} = \sum\limits_{0 \ne k \in {\mathbb{Z}^n}} {|\hat P(k)|\exp \left( {{C_{{\mathrm{G2}}}}{{\left| k \right|}^{\frac{1}{{1 + \eta }}}}} \right)}  \ll 1,\]
    thereby proving (ii). In contrast to case (ii), the proof of (iii) is essentially the same as that of Theorem \ref{VECT2}, thus completing the entire proof of Theorem \ref{VECYOUXIAN}.
 \begin{remark}
 	It is pointed out in the Diophantine small divisor Lemma \ref{VECSM} that the estimate in \eqref{VECJJJ} cannot be further reduced to the simple form presented in \eqref{VECYXG}. This represents a fundamental difference between the infinite-dimensional and finite-dimensional Gevrey cases, which ultimately leads to different regularity in the KAM theorems, specifically  Theorem \ref{VECT1} and (ii) in Theorem \ref{VECT3}.
 \end{remark}

\section{Further discussions}\label{VECSEC6}
\subsection{Some interesting connections with Corsi-Gentile-Procesi's KAM}\label{VECSEC61}
\renewcommand{\thefootnote}{\fnsymbol{footnote}}
Very recently, Corsi-Gentile-Procesi \cite{MR4781767} investigate the existence of full-dimensional tori in a mechanical system
\[H\left( {x,y} \right) = \frac{1}{2}y \cdot y + \varepsilon P\left( x \right)\]
 based on the tree formalism. The Hamiltonian systems there are analytic, whereas in the present paper, we consider $ C^\infty $ vector fields. Upon comparing the two, we discover some very interesting and essential connections, as kindly suggested by Professor Procesi.

 Firstly, both studies  observe that the respective approaches can be applied from infinite-dimension settings  to finite-dimension ones (for instance, \cite[Remark 3]{MR4781767} and Section \ref{VECAPP}). 
 
  Secondly, it should be noted that the vast majority of known analytic KAM techniques  cannot deal with the $ C^\infty $ case. One of the essential difficulties is that the Fourier expansion is not well-defined on a thickened torus.
  
    Thirdly, \cite{MR4781767} does not utilize the estimate based on $ |k|_\eta $ as in Lemma \ref{VECSM}.  Instead, it uses 
   \[|k|_{\star}:=\sum_{j \in \mathbb{Z}} \left|k_j\right|h_j, \;\; h_j=h_{-j} \in \mathbb{R}_{+} \forall j \in \mathbb{Z}, \;\; h_{j+1} \geqslant h_j \forall j \in \mathbb{Z}_{+} \]
   with 
   \[\limsup _{j \rightarrow+\infty} \frac{(\log (1+\langle j\rangle))^\sigma}{h_j}<+\infty \; \text { for some } \sigma>2\]
    to obtain a small divisor estimate as $ N \to +\infty $ for any $ \mu_1,\mu_2>0 $:
    \begin{equation}\label{VECC}
    	\sup _{\substack{\nu \in \mathbb{Z}_*^\infty},\; |k|_{\star} \leqslant N} \prod_{j \in \mathbb{Z}}\left(1+\langle j\rangle^{\mu_1}\left|\nu_j\right|^{\mu_2}\right) \leqslant K_1 \exp\left(K_2 N /(\ln N)^{\sigma-1}\right) \text{for some  universal } K_1, K_2>0,
    \end{equation}
     thereby relaxing as much as possible the spatial structure of the analytic perturbation $ P $ about $ k $. This is quite interesting because the small divisor estimate given in \eqref{VECC} is `critical' in the KAM iteration process, from the perspective of the Bruno condition or the weak Diophantine condition, see P\"oschel \cite{MR1037110} and the authors \cite{MR4836959}\footnote{The result in \cite{MR4836959} has been developed by the authors to the case of infinite-dimensional Hamiltonian systems with frequency-preserving, which corresponds to a Kolmogorov-type theorem.} for instance. In fact, the authors believe that the RHS of \eqref{VECC} could be further relaxed to 
     \[\exp \left( {{K_2}N{{\left( {\ln N} \right)}^{ - 1}}{{\left( {\ln \ln N} \right)}^{ - 1}} \cdots {{\Big( {\underbrace {\ln  \cdots \ln }_{\ell {\text{\;times}} }N } \Big)}^{1 - \sigma }}} \right)\]
    where $ \ell \in \mathbb{N}^ + $, but ensure the validity of the KAM iteration.  
    This would further improve the results in \cite{MR4781767} (by utilizing a weaker spatial structure).  However, the index $ \sigma>2 $ cannot degenerate to $ 2 $ (otherwise, at least the finite-dimensional KAM results would be invalid). As for the KAM regularity for the perturbation $ P $, \cite{MR4781767} requires (below, we give an intuitive statement for some universal $ C>0 $ without loss of generality)
     \begin{equation}\label{VECGUJI2}
     	| \hat{P} ( k ) |  \leqslant \exp \left( {-C\left( {\sum\limits_{j \in \mathbb{Z}} {|{k_j}|{{\left( {\ln \left\langle j \right\rangle } \right)}^\sigma }} } \right)} \right),
     \end{equation}
      while our Theorem \ref{VECT1}, based on the same Diophantine nonresonance (see Definition \ref{VECINDIO}), requires 
     \begin{equation}\label{VECGUJI3}
     	| \hat{P} ( k ) | \leqslant \exp \left( {-C{{\left( {\sum\limits_{j \in \mathbb{Z}} {|{k_j}|{{\left\langle j \right\rangle }^\eta }} } \right)}^{\frac{1}{{1 + \eta }}}}\ln \left( {\sum\limits_{j \in \mathbb{Z}} {|{k_j}|{{\left\langle j \right\rangle }^\eta }} } \right)} \right).
     \end{equation}
      Interestingly, these two conditions do not imply a direct relationship. We illustrate this with a few simple cases: (i) when restricted to finite dimensions, \eqref{VECGUJI2} and \eqref{VECGUJI3} correspond to classical analyticity and Gevrey regularity, respectively, with the latter being weaker; (ii) for a fixed $0\ne k \in \mathbb{Z}_*^{\infty}$, let $ j^* $ be the largest $ |j| $ such that $ k_j $ is  nonzero. If $ |k_{j^*}| $ tends to $ +\infty $ relatively large with respect to $ j^* $, and $ k_l $ is relatively small with respect to $ |k_{j^*}| $ for all $ l \ne j^* $, then obviously \eqref{VECGUJI3} is weaker; (iii) but if the components of $0\ne k \in \mathbb{Z}_*^{\infty}$ are uniformly bounded, for instance,
      \begin{equation}\label{VECK}
      	k = \left( { \ldots ,0,\mathop 1\limits_{ - {j_ * }} ,0, \ldots ,0,\mathop 1\limits_{{j_ * }} ,0, \ldots } \right),
      \end{equation}
       then as $ j^* $ tends to infinity, \eqref{VECGUJI2} could be  weaker, because  at least for $ k $ in \eqref{VECK},
      \[\sum\limits_{j \in \mathbb{Z}} {|{k_j}|{{\left( {\ln \left\langle j \right\rangle } \right)}^\sigma }}  = 2{\left( {\ln {j_ * }} \right)^\sigma } \leqslant {\left( {2j_ * ^\eta } \right)^{\frac{1}{{1 + \eta }}}}\ln \left( {2j_ * ^\eta } \right) = {\left( {\sum\limits_{j \in \mathbb{Z}} {|{k_j}|{{\left\langle j \right\rangle }^\eta }} } \right)^{\frac{1}{{1 + \eta }}}}\ln \left( {\sum\limits_{j \in \mathbb{Z}} {|{k_j}|{{\left\langle j \right\rangle }^\eta }} } \right).\]
       As a consequence, finding a more general technique to weaken the spatial structure in the infinite-dimensional case still  remains exploring. Actually, if one estimates in the manner of \eqref{VECC}, one can see that our current abstract $ m $-weighted KAM Theorem \ref{VECABSTRACT} is not directly  applicable. This is because,  during the iteration process, the contraction of the complex strip is of  an exponential type. As discussed in \cite{MR1037110,MR4836959}, if the contraction of the complex strip is constructed to be slower, the KAM iteration can be established (but the frequency must not be extremely rational). Therefore, the approach presented in this paper could be further developed---it might be applicable to the spatial structure in \cite{MR4781767} (in other words, it might be applicable to certain weaker nonresonance)---which will be the subject of our future research.

\subsection{Two alternative approaches toward the Qualitative $ C^\infty $ type KAM Theorem \ref{VECT33} and minimizing regularity}\label{VECSEC62}
\renewcommand{\thefootnote}{\fnsymbol{footnote}}
To establish the KAM theorem with $ C^\infty $ initial regularity, we have constructed two nonresonance different from Bourgain's form (see Definition \ref{VECINDIO}) in the proofs of Theorems \ref{VECT2} to \ref{VECT33}, in a P\"oschel's fashion \cite{MR1037110}. Following Professor Procesi's suggestions, we could also construct nonresonance in the manner of Bourgain \cite{MR2180074}, Chierchia-Perfetti \cite{MR1317707} and etc. Here, we provide two alternative approaches to the Qualitative $ C^\infty $ type KAM Theorem \ref{VECT33} (which are also valid for Theorem \ref{VECT2}), one utilizes Bourgain's nonresonance, while another does not strongly rely on the specific construction of nonresonance. Thereby, we have at least \textit{three} different approaches to establish the infinite-dimensional KAM theory with $ C^\infty $ initial regularity. To enhance readers' comprehension of Theorem \ref{VECT33}, we delve into a more detailed quantification\footnote{To preserve as much consistency with the first version of this paper \cite{arXiv:2306.08211} as possible, we prefer to keep the original \textit{qualitative} Theorem \ref{VECT33} in Section \ref{VECSTEATE}, while providing a more \textit{quantitative} version in this section. Indeed, some of the estimates presented here are somewhat `rougher' than those in Theorem \ref{VECT1}, and there is potential for further refinement, particularly in certain explicit cases.} in this section.  Additionally, we provide Remark \ref{VECRE61} to further illustrate the  minimization idea.

\begin{theorem}[\textbf{Quantitative $ C^\infty $ type  KAM}]\label{VECT333}
	For almost all $ \omega \in \mathbb{R}^{\mathbb{Z}} $ and any given approximation function $ \Delta $ larger than any polynomial's type, there exists a uniform approximation function $ u $ depending on $ \Delta $,  as long as  the perturbation $ P $ is sufficiently small in the sense that
	\[\sum\limits_{0 \ne k \in \mathbb{Z}_ * ^\infty } {|\hat P(k)|\Delta({\left\| k \right\|}_{u})}  \ll 1,\]
	where $ {\left\| k \right\|_u}: = \sum\nolimits_{j \in \mathbb{Z}} {u\left( {\left\langle j \right\rangle } \right)\left| {{k_j}} \right|}  $,  the KAM conjugacy in Theorem \ref{VECT1} holds. 
	
	Specifically, if $ \Delta $ is at most Gevrey, i.e., there exists some $0< a <1 $ such that $ \Delta(x)=\mathcal{O}\left(\exp\left(x^a\right)\right) $, then the approximation function $ u $ can be chosen as 	$ u(x)=\mathcal{O}^{\#}\left({{{\log }_x^{-1}}\Delta \left( x \right)}\right)$. Consequently, for any given set of  such approximation functions $ \{\Delta_i\}_{1\leqslant i \leqslant \mathscr{N}} $, the KAM regularity can be `minimized' to 
\begin{equation}\label{VECMIN}
	\sum\limits_{0 \ne k \in \mathbb{Z}_*^\infty } {|\hat P(k)|{{\tilde \Delta }_k}}  \ll 1,
\end{equation}
	where
	\[{{\tilde \Delta }_k} = \mathop {\min }\limits_{1 \leqslant i \leqslant \mathscr{N}} {\Delta _i}\left( {{{\left\| k \right\|}_{\log _x^{ - 1}{\Delta _i}\left( x \right)}}} \right),\;\;0 \ne k \in \mathbb{Z}_*^\infty.\]
\end{theorem}
\begin{remark}\label{VECRE61}
It should be emphasized that the pursuit of minimizing regularity in \eqref{VECMIN} is significant, as it is not always the case that a weaker $ \Delta $ (indicating slower growth) leads to a correspondingly weaker KAM regularity. This is because the spatial structure plays a role as well, such as the form of the norm $ \|k\|_u $. Indeed, as demonstrated in Section \ref{VECSEC61}, for certain specific $0\ne k  \in \mathbb{Z}_*^\infty$ (there are infinitely many), a stronger $ \Delta $ can result in weaker regularity (though this is not the case for all $0\ne k  \in \mathbb{Z}_*^\infty$). We will now present an explicit example with $ \mathscr{N}=2 $ based on the KAM regularity outlined in Theorem \ref{VECT333}.

Set the approximation functions with a sufficiently large number $ C>0 $ as
\begin{equation}\label{VECD12}
	{\Delta _1}\left( x \right) = \exp \left( {C{{\left( {\ln x} \right)}^a}} \right),\;\;a > 1\text{\;\;\;and\;\;\;}{\Delta _2}\left( x \right) = \exp \left( {C\left( {\ln x} \right)\left( {\ln \ln x} \right)} \right).
\end{equation}
It is evident that both of them do not exceed Gevrey's type. Therefore, we can choose approximation functions
\[{u_1}\left( x \right) = \exp \left( {{x^{\frac{1}{{a - 1}}}}} \right),\;\;{u_2}\left( x \right) = \exp \left( {\exp x} \right)\]
such that Theorem \ref{VECT333} holds. From an intuitive perspective, we can require that the Fourier coefficients of the perturbation satisfy the decay conditions
\[|\hat f\left( k \right)| \leqslant {\Delta _1}\left( {C{{\left\| k \right\|}_{{u_1}}}} \right) ^{-1}= \exp \left( { - C{{\ln }^a}\left( {\sum\limits_{j \in \mathbb{Z}} {|{k_j}|\exp \left( {{{\left\langle j \right\rangle }^{\frac{1}{{a - 1}}}}} \right)} } \right)} \right): = \tilde \Delta _k^1\]
and
\begin{align*}
|\hat f\left( k \right)| &\leqslant {\Delta _2}\left( {C{{\left\| k \right\|}_{{u_2}}}} \right) ^{-1}\\
& = \exp \left( { - C\ln \left( {\sum\limits_{j \in \mathbb{Z}} {|{k_j}|\exp \left( {\exp \left( {\left\langle j \right\rangle } \right)} \right)} } \right)\left( {\ln \ln \left( {\sum\limits_{j \in \mathbb{Z}} {|{k_j}|\exp \left( {\exp \left( {\left\langle j \right\rangle } \right)} \right)} } \right)} \right)} \right)\\
: &= \tilde \Delta _k^2,
\end{align*}
simultaneously, i.e., $ |\hat f\left( k \right)| \leqslant {{\tilde \Delta }_k} := \min \{ {\tilde \Delta _k^1,\tilde \Delta _k^2} \} $. Let us discuss several simple cases---\textit{note that these are far from all the possible situations}:
\begin{itemize}
\item[(I)] In the finite-dimensional case, the regularity corresponding to $ {\tilde \Delta _k^2} $ is weaker than that of $ {\tilde \Delta _k^1} $;

\item[(II)] In the infinite-dimensional case, consider that a fixed $ 0\ne k \in \mathbb{Z}_*^\infty$ (infinitely many) has the form 
\[	k = \left( { \ldots ,0,\mathop {k_{-n}}\limits_{ - {n }} ,k_{-n+1}, \ldots ,k_{n-1},\mathop {k_n}\limits_{{n }} ,0, \ldots } \right),\]
provided that $ n \in \mathbb{N}^+ $ is uniformly bounded. Then, it is essentially the same as the  finite-dimensional case, i.e., the regularity corresponding to $ {\tilde \Delta _k^2} $ is weaker than that of $ {\tilde \Delta _k^1} $;
\item[(III)] In the infinite-dimensional case, consider that a fixed $ 0\ne k \in \mathbb{Z}_*^\infty$ (infinitely many) has the form 
\[	k = \left( { \ldots ,0,\mathop {k_{-j_*}}\limits_{ - {j_ * }} ,0, \ldots ,0,\mathop {k_{j_*}}\limits_{{j_ * }} ,0, \ldots } \right),\]
where $ j_* \to +\infty$ as $ \min \{ {{{\left\| k \right\|}_{{u_1}}},{{\left\| k \right\|}_{{u_2}}}} \} \to  + \infty  $. Next, we only need to analyze the order within the exponential functions in $ \tilde \Delta _k^1 $ and $ \tilde \Delta _k^2$, i.e., 
\[{{{\ln }^a}\left( {\sum\limits_{j \in \mathbb{Z}} {|{k_j}|\exp \left( {{{\left\langle j \right\rangle }^{\frac{1}{{a - 1}}}}} \right)} } \right)}\text{\;and\;}
 {\ln \left( {\sum\limits_{j \in \mathbb{Z}} {|{k_j}|\exp \left( {\exp \left( {\left\langle j \right\rangle } \right)} \right)} } \right)\left( {\ln \ln \left( {\sum\limits_{j \in \mathbb{Z}} {|{k_j}|\exp \left( {\exp \left( {\left\langle j \right\rangle } \right)} \right)} } \right)} \right)}.\]
In other words, we only need to compare in this case the order of 
\begin{equation}\label{VECBJ2}
{\ln ^a}\Big( {|2{k_{{j_ * }}}|\exp \big( {j_ * ^{\frac{1}{{a - 1}}}} \big)} \Big)\text{\;\;and\;\;}\ln \left( {2|{k_{{j_ * }}}|\exp \left( {\exp \left( {{j_ * }} \right)} \right)} \right) \cdot \ln \ln \left( {2|{k_{{j_ * }}}|\exp \left( {\exp \left( {{j_ * }} \right)} \right)} \right).
\end{equation}
\begin{itemize}
\item[(i)] If $ \frac{{\ln |{k_{{j_ * }}}|}}{{j_ * ^{1/\left( {a - 1} \right)}}} = \mathcal{O}\left( 1 \right) $, then for \eqref{VECBJ2}, we have
\[{\rm LHS} = {\mathcal{O}^\# }\left( {j_ * ^{\frac{a}{{a - 1}}}} \right) = o\left( {\exp \left( {{j_ * }} \right)} \right),\;\;{\rm RHS} \geqslant \ln \left( {\exp \left( {\exp \left( {{j_ * }} \right)} \right)} \right) = \exp \left( {{j_ * }} \right).\]
This implies that the regularity corresponding to $ {\tilde \Delta _k^1} $ is weaker than that of $ {\tilde \Delta _k^2} $;
\item[(ii)] If $ \frac{{j_ * ^{1/\left( {a - 1} \right)}}}{{\ln |{k_{{j_ * }}}|}} = o\left( 1 \right) $ and $ \frac{{\ln |{k_{{j_ * }}}|}}{{\exp \left( {{j_ * }/a} \right)j_ * ^{1/a}}} = o\left( 1 \right) $, then for \eqref{VECBJ2}, we have
\[{\rm LHS} = {\mathcal{O}^\# }\left( {{{\ln }^a}|{k_{{j_ * }}}|} \right) = o\left( {\exp \left( {{j_ * }} \right){j_ * }} \right),\;\;{\rm RHS} = {\mathcal{O}^\# }\left( {\exp \left( {{j_ * }} \right){j_ * }} \right).\]
This implies that the regularity corresponding to $ {\tilde \Delta _k^1} $ is weaker than that of $ {\tilde \Delta _k^2} $;
\item[(iii)] If $ \frac{{\ln |{k_{{j_ * }}}|}}{{\exp \left( {{j_ * }/a} \right)j_ * ^{1/a}}} = {\mathcal{O}^\# }\left( 1 \right) $,  then for \eqref{VECBJ2}, we have
\[{\rm LHS} = {\mathcal{O}^\# }\left( {{{\ln }^a}|{k_{{j_ * }}}|} \right),\;\;{\rm RHS} = {\mathcal{O}^\# }\left( {\exp \left( {{j_ * }} \right){j_ * }} \right) = {\mathcal{O}^\# }\left( {{{\ln }^a}|{k_{{j_ * }}}|} \right).\]
This implies that the regularity corresponding to $ {\tilde \Delta _k^1} $ is `similar' to that of $ {\tilde \Delta _k^2} $ (if a comparison of the strength of regularity is indeed necessary, one also needs to delve into the specific values of the constant $ C>0 $ in \eqref{VECD12}; however, for the sake of simplicity in this context, we will treat them as being the same);

\item[(iv)] If $ \frac{{\ln |{k_{{j_ * }}}|}}{{\exp \left( {{j_ * }} \right)}} = \mathcal{O}\left( 1 \right) $ and $ \frac{{\exp \left( {{j_ * }/a} \right)j_ * ^{1/a}}}{{\ln |{k_{{j_ * }}}|}} = o\left( 1 \right) $, then for \eqref{VECBJ2}, we have
\[{\rm LHS} = {\mathcal{O}^\# }\left( {{{\ln }^a}|{k_{{j_ * }}}|} \right),\;\;{\rm RHS} = {\mathcal{O}^\# }\left( {\exp \left( {{j_ * }} \right){j_ * }} \right) = o\left( {{{\ln }^a}|{k_{{j_ * }}}|} \right).\]
This implies that the regularity corresponding to $ {\tilde \Delta _k^2} $ is weaker than that of $ {\tilde \Delta _k^1} $;
\item[(v)] If $ \frac{{\exp \left( {{j_ * }} \right)}}{{\ln |{k_{{j_ * }}}|}} = o\left( 1 \right) $,  then for \eqref{VECBJ2}, we have
\[{\rm LHS} = {\mathcal{O}^\# }\left( {{{\ln }^a}|{k_{{j_ * }}}|} \right),\;\;{\rm RHS} = {\mathcal{O}^\# }\left( {\ln |{k_{{j_ * }}}| \cdot \ln \ln |{k_{{j_ * }}}|} \right) = o\left( {{{\ln }^a}|{k_{{j_ * }}}|} \right).\]
This implies that the regularity corresponding to $ {\tilde \Delta _k^2} $ is weaker than that of $ {\tilde \Delta _k^1} $;
\end{itemize}

\item[(IV)] In the infinite-dimensional case, consider that a fixed $ 0\ne k \in \mathbb{Z}_*^\infty$ (infinitely many) has the form 
\[k = \left( { \ldots ,0,\mathop {\frac{1}{2}}\limits_{ - {j_*}} ,0, \ldots ,0,\mathop {{k_{ - \ln {j_*}}}}\limits_{ - \ln {j_*}} ,0, \ldots ,0,\mathop {{k_{\ln {j_*}}}}\limits_{\ln {j_*}} ,0, \ldots ,0,\mathop {\frac{1}{2}}\limits_{{j_*}} ,0, \ldots } \right),\]
where $ j_* \to +\infty$ as $ \min \{ {{{\left\| k \right\|}_{{u_1}}},{{\left\| k \right\|}_{{u_2}}}} \} \to  + \infty  $.  In this case, we only need to compare in this case the order of 
\[{\ln ^a}\left( {\exp \big( {j_ * ^{\frac{1}{{a - 1}}}} \big) + |{k_{\ln {j_*}}}|\exp \left( {{{\left( {\ln {j_*}} \right)}^{\frac{1}{{a - 1}}}}} \right)} \right)\]
and
\[\ln \left( {\exp \left( {\exp \left( {{j_ * }} \right)} \right) + |{k_{\ln {j_*}}}|\exp \left( {{j_*}} \right)} \right) \cdot \ln \ln \left( {\exp \left( {\exp \left( {{j_ * }} \right)} \right) + |{k_{\ln {j_*}}}|\exp \left( {{j_*}} \right)} \right).\]
\textit{This is much more complex than the case in (III).}
 \begin{itemize}
\item [(i)] ......
\item [.......]
 \end{itemize}

\item[(V)]  ......

\end{itemize}

In conclusion, it is evident that even in the most straightforward minimization case (i.e., $ \mathscr{N}=2 $), we must employ  complicated asymptotic analysis to determine the strength or weakness of regularity, even in the most basic cases. This necessity intensifies when faced with slightly more intricate situations (e.g., arbitrary $ 0 \ne k \in \mathbb{Z}_*^\infty $). Therefore, the idea of minimization is extremely significant in quantitatively reducing KAM regularity within almost periodic settings.

\end{remark}
\begin{proof}[Proof of Theorem \ref{VECT333}]
	Before proceeding with the analysis, we first note that as long as $ \Delta $ does not exceed Gevrey's type, the balancing sequence $\{d_j\}_{j \in \mathbb{N}}$ can be appropriately chosen (growing at most exponentially) to satisfy the boundedness condition in (II) provided in the Abstract $ m $-weighted KAM Theorem \ref{VECABSTRACT}. Therefore, in the $C^\infty$ case, the most critical aspect is achieving a balance between the small divisor estimate and regularity. An important idea is that by enhancing the weight $ u $ within the norm  $ \|k \|_u $, we can reduce the influence of small divisors. Due to the fact that we have carried out detailed quantitative constructions in the proofs of all previous theorems, we will slightly simplify the analysis here.

\textbf{Approach I: }
For the given `super-polynomial' approximating function $ \Delta $, we first leave the approximating function $ u $ undetermined and define the norm of $ k \in \mathbb{Z}_*^\infty $ as  $ {\left\| k \right\|_u}: = \sum\nolimits_{j \in \mathbb{Z}} {\left| {{k_j}} \right|u\left( {\left\langle j \right\rangle } \right)}   $. Let $ m $ denote the number of nonzero components of $ k $. Note that if $ {\left\| k \right\|_u} \leqslant N$, we have at least that $ m \leqslant {u^{ - 1}}\left( N \right) $, because
\[1 \cdot u\left( m \right) \leqslant \sum\nolimits_{j \in \mathbb{Z}} {\left| {{k_j}} \right|u\left( {\left\langle j \right\rangle } \right)}  = {\left\| k \right\|_u} \leqslant N.\]
We shall emphasize that this is not very precise in some cases, see Lemma \ref{VECSM} for instance. However, here we only need to qualitatively obtain an upper bound for $ m $. Next, we assume that $ \omega $ satisfies the following Bourgain's nonresonance for which almost all $ \omega \in \mathbb{R}^\mathbb{Z} $ hold (see Definition \ref{VECINDIO})
	\begin{equation}\notag 
	\left| {k  \cdot \omega } \right| > \gamma^* \prod\limits_{j \in \mathbb{Z}} {\frac{1}{{\left( {1 + {{\left| {{k _j}} \right|}^\mu }{{\left\langle j \right\rangle }^\mu }} \right)}}} ,\;\;\forall 0\ne k  \in {\mathbb{Z}_*^\infty},\;\; \gamma^*>0.
\end{equation}
Note that if $ 	u\left( x \right) \geqslant x $,
 we have
\begin{align}
	\mathop {\sup }\limits_{0 < {{\left\| k \right\|}_u} \leqslant N} {\left| {k \cdot \omega } \right|^{ - 1}} &\leqslant {\gamma ^*}^{ - 1}\prod\limits_{j \in \mathbb{Z}} {\left( {1 + {{\left| {{k_j}} \right|}^\mu }{{\left\langle j \right\rangle }^\mu }} \right)} \notag \\
	& = {\gamma ^*}^{ - 1}\exp \left( {\sum\limits_{j \in \mathbb{Z}} {\ln \left( {1 + {{\left| {{k_j}} \right|}^\mu }{{\left\langle j \right\rangle }^\mu }} \right)} } \right)\notag\\
	& = \exp \left( {\mathcal{O}\left( {\sum\limits_{j \in \mathbb{Z}} {\ln \left( {1 + \left| {{k_j}} \right|\left\langle j \right\rangle } \right)} } \right)} \right)\notag\\
\label{VECBOU1}	& \leqslant \exp \left( {\mathcal{O}\left( {\sum\limits_{j \in \mathbb{Z}} {\ln \left( {1 + \left| {{k_j}} \right|u\left\langle j \right\rangle } \right)} } \right)} \right)\\
	&  \leqslant \exp \left( {\mathcal{O}\left( {\# \left\{ {0 \ne k \in \mathbb{Z}_ * ^\infty :{{\left\| k \right\|}_u} \leqslant N} \right\}} \right) \cdot \mathop {\sup }\limits_{0 < {{\left\| k \right\|}_u} \leqslant N} \ln \left( {1 + \sum\limits_{s \in \mathbb{Z}} {\left| {{k_s}} \right|u\left( {\left\langle s \right\rangle } \right)} } \right)} \right)\notag\\
\label{VECBOU2}	&  \leqslant \exp \left( {\mathcal{O}\left( {{u^{ - 1}}\left( N \right)} \right) \cdot \mathop {\sup }\limits_{0 < {{\left\| k \right\|}_u} \leqslant N} \ln \left( {1 + {{\left\| k \right\|}_u}} \right)} \right)\\
\label{VECSMA} 	& = \exp \left( {\mathcal{O}\left( {{u^{ - 1}}\left( N \right)\ln \left( {1 + N} \right)} \right)} \right),
\end{align}
where \eqref{VECBOU1} uses $ 	u\left( x \right) \geqslant x $, and  \eqref{VECBOU2} uses $ m \leqslant {u^{ - 1}}\left( N \right) $. To  eliminate the growth caused by small divisors in \eqref{VECSMA} by the weight $ m $ with $ m_k= \Delta \left( {{{\left\| k \right\|}_u}} \right)$, we impose that 
\[\exp \left( {{u^{ - 1}}\left( N \right)\ln N} \right) \leqslant \mathcal{O}\left( {\Delta \left( {{{\left\| k \right\|}_u}} \right){|_{{{\left\| k \right\|}_u} = N}}} \right) = \mathcal{O}\left( {\Delta \left( N \right)} \right),\]
and this will finally implies  the boundedness condition in (II).  Hence, $ {u^{ - 1}}\left( N \right)\ln N = \mathcal{O}\left( {\ln \left( {\Delta \left( N \right)} \right)} \right)$, indicating that taking $ u\left( x \right) = {\mathcal{O}^\# }\left( {\log _x^{ - 1}\Delta \left( x \right)} \right)\geqslant x $ is sufficient. As a consequence, we obtain the desired KAM regularity as
\[\sum\limits_{0 \ne k \in \mathbb{Z}_*^\infty } {|\hat P(k)|{m_k}}  = \sum\limits_{0 \ne k \in \mathbb{Z}_*^\infty } {|\hat P(k)|\Delta \left( {{{\left\| k \right\|}_u}} \right)}  \ll 1.\]

\textbf{Approach II: } Similar to that in Approach I, we still leave the approximating function $ u $ undetermined and introduce the norm of $ k \in \mathbb{Z}_*^\infty $ as  $ {\left\| k \right\|_u}: = \sum\nolimits_{j \in \mathbb{Z}} {\left| {{k_j}} \right|u\left( {\left\langle j \right\rangle } \right)}   $. Hence, the number of nonzero components of $ k $,  denoted as $ m $,  at least satisfies  $ m \leqslant {u^{ - 1}}\left( N \right) $. In the subsequent analysis, we do not employ any previously introduced infinite-dimensional nonresonance. Instead, we rely on a crucial fact: for any fixed $ 0 \ne k \in \mathbb{Z}_*^\infty $, we do not actually encounter the infinite case for $ \left| {k \cdot \omega } \right| $. Observe that in the $n $-dimensional case, almost all $ \omega\in \mathbb{R}^n $ satisfy a Diophantine estimate of the type
\[\left| {k \cdot \omega } \right| \geqslant \frac{{{\gamma ^ * }}}{{{{\left| k \right|}^n}}},\;\;\forall 0 \ne k \in {\mathbb{Z}^n},\;\;{\gamma ^ * } > 0,\]
where $\gamma^*$ is independent of $ n $. Then, for any $ 0 \ne k \in \mathbb{Z}_*^\infty  $ with $ {\left\| k \right\|_u} \leqslant N$, it universally holds that
\begin{equation}\notag 
	\left| {k \cdot \omega } \right| \geqslant \frac{{{\gamma ^ * }}}{{{{\left| k \right|}^m}}} \geqslant \frac{{{\gamma ^ * }}}{{\left\| k \right\|_u^m}} \geqslant \frac{{{\gamma ^ * }}}{{{N^m}}} \geqslant \frac{{{\gamma ^ * }}}{{{N^{{u^{ - 1}}\left( N \right)}}}} = {\gamma ^ * }\exp \left( { - {u^{ - 1}}\left( N \right)\ln N} \right).
\end{equation}
Now, we need to balance the small divisor estimate 
\begin{equation}\label{VECGJ2}
	\mathop {\sup }\limits_{0 < {{\left\| k \right\|}_u} \leqslant N} {\left| {k \cdot \omega } \right|^{ - 1}} \leqslant {\gamma ^ * }^{ - 1}\exp \left( {{u^{ - 1}}\left( N \right)\ln N} \right)
\end{equation}
and the weight $ m $, i.e., $ m_k= \Delta \left( {{{\left\| k \right\|}_u}} \right)$. We hope that the weight can eliminate the growth caused by small divisors in \eqref{VECGJ2}; therefore, recalling the boundedness condition in (II), let 
\[\exp \left( {{u^{ - 1}}\left( N \right)\ln N} \right) \leqslant \mathcal{O}\left( {\Delta \left( {{{\left\| k \right\|}_u}} \right){|_{{{\left\| k \right\|}_u} = N}}} \right) = \mathcal{O}\left( {\Delta \left( N \right)} \right).\]
Then, $ {u^{ - 1}}\left( N \right)\ln N = \mathcal{O}\left( {\ln \left( {\Delta \left( N \right)} \right)} \right)$, indicating that taking $ u\left( x \right) = {\mathcal{O}^\# }\left( {\log _x^{ - 1}\Delta \left( x \right)} \right) $ will meet the requirements. Finally, the KAM regularity is expressed as
\[\sum\limits_{0 \ne k \in \mathbb{Z}_*^\infty } {|\hat P(k)|{m_k}}  = \sum\limits_{0 \ne k \in \mathbb{Z}_*^\infty } {|\hat P(k)|\Delta \left( {{{\left\| k \right\|}_u}} \right)}  \ll 1,\]
as desired.
\end{proof}

 \section*{Acknowledgements} 
This work was supported in part by the National Natural Science Foundation of China (Grant Nos. 12071175 and 12471183). The authors express their profound gratitude to Professor Michela Procesi for her meticulous review of the first version of this work \cite{arXiv:2306.08211} and for providing insightful suggestions that significantly enhanced its quality.

\end{document}